\documentclass{amsart} \usepackage{mathrsfs,epsfig}

\theoremstyle{plain} \newtheorem{thm}{Theorem}[section]
\newtheorem{lemma}[thm]{Lemma} \newtheorem{conjecture}[thm]{Conjecture}
 \newtheorem{prop}[thm]{Proposition}

\theoremstyle{definition} \newtheorem{definition}[thm]{Definition}
\newtheorem{remark}[thm]{Remark} 
\newtheorem{question}[thm]{Question} 

\catcode`\@=11 
\def\asequal{\mathrel{\mathpalette\@mvereq{\hbox{\sevenrm a.s.}}}}
\def\@mvereq#1#2{\lower.5\p@\vbox{\baselineskip\z@skip\lineskip1.5\p@
    \ialign{$\m@th#1\hfil##\hfil$\crcr#2\crcr=\crcr}}}
\def\acong{\mathrel{\mathpalette\@avereq\sim}} 
\def\@avereq#1#2{\lower.5\p@\vbox{\baselineskip\z@skip\lineskip-.5\p@
    \ialign{$\m@th#1\hfil##\hfil$\crcr#2\crcr\longrightarrow\crcr}}}
\def\newdot{{\kern.8pt\cdot\kern.8pt}}
\font\sevenrm=cmr7

\newcommand{\E}{\mathbb{E}}\renewcommand{\H}{\mathbb{H}} 
\newcommand{\N}{\mathbb{N}}
\newcommand{\R}{\mathbb{R}} 
\renewcommand{\P}{\mathbb{P}}

\newcommand{\SB}{{\mathscr B}}  
\newcommand{\SD}{{\mathscr D}}  
\newcommand{\SF}{{\mathscr F}}
\newcommand{\SH}{{\mathscr H}} 

\DeclareMathOperator{\inv}{inv} \DeclareMathOperator{\pr}{pr}

\def\mathpal#1{\mathop{\mathchoice{\text{\rm #1}}%
   {\text{\rm #1}}{\text{\rm #1}}%
   {\text{\rm #1}}}\nolimits} \def\Ric{\mathpal{Ric}}
\def\Sect{\mathpal{Sect}} \def\dist{\mathpal{dist}}
  \def\pr{\mathpal{pr}}
 \def\di{\displaystyle}

\def\f{\frac} \def\a{\alpha } \def\b{\beta }  \def\d{\delta }
\def\e{\varepsilon }  \def\g{\gamma } \def\l{\lambda }
 \def\Om{\Omega } \def\om{\omega } \def\p{\partial }
\def\s{\sigma }

\begin{document}

\title[Existence of non-trivial harmonic functions] {Existence of non-trivial
  harmonic functions\\ on Cartan-Hadamard manifolds\\ of unbounded curvature}

\author[M. Arnaudon]{Marc Arnaudon} \address{D\'epartement de
  math\'ematiques\hfill\break\indent Universit\'e de Poitiers, T\'el\'eport 2
  - BP 30179\hfill\break\indent F--86962 Futuroscope Chasseneuil Cedex,
  France} \email{arnaudon@math.univ-poitiers.fr}

\author[A. Thalmaier]{Anton Thalmaier} \address{Institute of Mathematics,
  University of Luxembourg\hfill\break\indent 162A, avenue de la
  Fa\"{\i}encerie\hfill\break\indent L--1511 Luxembourg, Grand-Duchy of
  Luxembourg} \email{anton.thalmaier@uni.lu}

\author[S. Ulsamer]{Stefanie Ulsamer} \address{NWF I --
  Mathematik\hfill\break\indent Universit\"at Regensburg\hfill\break\indent
  D--93040 Regensburg, Germany} \email{stefanie.ulsamer@d-fine.de}


\date{\today}

%
%

\begin{abstract}\noindent
  The Liouville property of a complete Riemannian manifold $M$ (i.e., the
  question whether there exist non-trivial bounded harmonic functions on $M$)
  attracted a lot of attention. For Cartan-Hadamard manifolds the role of
  lower curvature bounds is still an open problem. We discuss examples of
  Cartan-Hadamard manifolds of unbounded curvature where the limiting angle of
  Brownian motion degenerates to a single point on the sphere at infinity, but
  where nevertheless the space of bounded harmonic functions is as rich as in
  the non-degenerate case.  
  To see the full boundary the point at infinity has to be blown up in a 
  non-trivial way. 
  Such examples indicate that the situation
  concerning the famous conjecture of Greene and Wu about existence of
  non-trivial bounded harmonic functions on Cartan-Hadamard manifolds is much
  more complicated than one might have expected.
\end{abstract}

\maketitle
\setcounter{tocdepth}{1}
\tableofcontents

\section{Introduction}\label{Section1}
\setcounter{equation}0

The study of harmonic functions on complete Riemannian manifolds, i.e., the
solutions of the equation $\Delta u=0$ where $\Delta$ is the Laplace-Beltrami
operator, lies at the interface of analysis, geometry and stochastics.
Indeed, there is a deep interplay between geometry, harmonic function theory,
and the long-term behaviour of Brownian motion.  Negative curvature amplifies
the tendency of Brownian motion to move away from its starting point and, if
topologically possible, to wander out to infinity. On the other hand,
non-trivial asymptotic properties of Brownian paths for large time correspond
with non-trivial bounded harmonic functions on the manifold.

There is plenty of open questions concerning the richness of certain spaces of
harmonic functions on Riemannian manifolds.  Even in the case of simply
connected negatively curved Riemannian manifolds basic questions are still
open.  For instance, there is not much known about the following problem posed
by Wu in 1983.

\begin{question}[cf.~Wu \cite{Wu:83} p.\,139]
   \label{Conj:Wu}
   If $M$ is a simply-connected complete Riemannian manifold with sectional
   curvature $\leq{-c}<0$, do there exist $n$ bounded harmonic functions
   {\rm(}$n=\dim M${\rm)} which give global coordinates on $M$?
\end{question}

It should be remarked that, under the given assumptions, it is even not known in
general whether there exist non-trivial bounded harmonic functions at all.
This question is the content of the famous Greene-Wu conjecture which asserts
existence of non-constant bounded harmonic functions under slightly more
precise curvature assumptions.

\begin{conjecture}[cf.~Greene-Wu \cite{Greene-Wu:79} p.\,767]
   \label{Conj:Greene-Wu}
   Let $M$ be a simply-connected complete Riemannian manifold of non-positive
   sectional curvature and $x_0\in M$ such that
   $$\hbox{\rm Sect}^M_x\leq -c\,r(x)^{-2}\quad \text{for all }\,x\in
   M\setminus K$$
   for some $K$ compact, $c>0$ and $r=\hbox{\rm dist}(x_0,\newdot)$.  
   Then $M$ carries non-constant bounded harmonic functions.
\end{conjecture}

Concerning Conjecture \ref{Conj:Greene-Wu} substantial progress has been made
since the pioneering work of Anderson \cite{Anderson:83}, Sullivan
\cite{Anderson:83}, and Anderson-Schoen \cite{Anderson-Schoen:85}.
Nevertheless, the role of lower curvature bounds is far from being understood.

From a probabilistic point of view, Conjecture \ref{Conj:Greene-Wu} concerns
the eventual behaviour of Brownian motion on these manifolds as time goes to
infinity.  Indeed, for any Riemannian manifold, we have the following
probabilistic characterization.

\begin{lemma}
  For a Riemannian manifold $(M,g)$ the following two conditions are
  equivalent:
\begin{enumerate}
\item[\rm i)] There exist non-constant bounded harmonic functions on $M$.
  
\item[\rm ii)] {\rm BM} has non-trivial exit sets, i.e., if $X$ is a Brownian
  motion on $M$ then there exist open sets $U$ in the $1$-point
  compactification $\hat M$ of $M$ such that
  $$
  \P\{X_t\in U\hbox{\ eventually}\}\not=0\hbox{\ or\ }1.
  $$
\end{enumerate}
\end{lemma}

More precisely, Brownian motion $X$ on $M$ may be realized on the space $C(\R_+,\hat
M)$ of continuous paths with values in the $1$-point compactification $\hat M$
of $M$, equipped with the standard filtration $\SF_t=\sigma\{X_s=\pr_s\vert
s\leq t\}$ generated by the coordinate projections $\pr_s$ up to time~$t$.  Let
$\zeta=\sup\{t>0: X_t\in M\}$ be the lifetime of $X$ and let
$\SF_{\hbox{\sevenrm inv}}$ denote the shift-invariant $\sigma$-field on
$C(\R_+,\tilde M)$.  Then there is a canonical isomorphism between the space
$\SH_{\hbox{\sevenrm b}}(M)$ of bounded harmonic functions on $M$ and the set
$b\SF_{\hbox{\sevenrm inv}}$ of bounded $\SF_{\hbox{\sevenrm inv}}$-measurable
random variables up to equivalence, given as follows:
\begin{equation}
   \label{Lemma:exitset}
\SH_{\hbox{\sevenrm b}}(M) \acong b\SF_{\hbox{\sevenrm inv}}/_\sim,\quad
   u\longmapsto \lim_{t\uparrow\zeta}(u\circ X_t).
\end{equation}
(Bounded shift-invariant random variables are considered as equivalent, if
they agree $\P_x$-a.e., for each $x\in M$.)  Note that the isomorphism
\eqref{Lemma:exitset} is well defined by the martingale convergence theorem,
and that the inverse map to \eqref{Lemma:exitset} is given by taking
expectations:
\begin{equation}
   \label{Lemma:exitsetinv}
b\SF_{\hbox{\sevenrm inv}}/_\sim \ni H \longmapsto u\in \SH_{\hbox{\sevenrm b}}(M)\ \text{ where }u(x):=\E_x[H]. 
\end{equation}
In particular,
$$u(x):=\P_x\{X_t\in U\hbox{\ eventually}\}$$
is a bounded harmonic function
on $M$, and {\it non-constant\/} if and only if $U$ is a {\it non-trivial\/}
exit set.

Now let $(M,g)$ be a \textit{Cartan-Hadamard manifold}, i.e., a simply
connected complete Riemannian manifold of non-positive sectional curvature.
(All manifolds are supposed to be connected).  In terms of the exponential map
$\exp_{x_0}\colon T_{x_0}M\acong M$ at a fixed base point $x_0\in M$, we
identify $\rho\colon \R^n\cong T_{x_0}M\acong M$.  Via pullback of the metric
on $M$, we get an isometric isomorphism $(M,g)\cong(\R^n,\rho^*g)$, and in
particular, $M\setminus\{x_0\}\cong{]0,\infty[}\times S^{n-1}$.  In terms of
such global polar coordinates, Brownian motions $X$ on $M$ may be decomposed
into their radial and angular part,
$$X_t=(r(X_t),\vartheta(X_t))$$
where $r(X_t)=\dist(x_0,X_t)$ and where
$\vartheta(X_t)$ takes values in $S^{n-1}$.

For a Cartan-Hadamard manifold $M$ of dimension $n$ there is a natural
geometric boundary, the \textit{sphere at infinity} $S_\infty(M)$, such that
$M\cup S_\infty(M)$ equipped with the \textit{cone topology} is homeomorphic
to the unit ball $B\subset \R^d$ with boundary $\partial B = S^{d-1}$,
cf.~\cite{Eberlein-ONeill:73}, \cite{Bishop-ONeill:68}. In terms of polar
coordinates on $M$, a sequence $(r_n,\vartheta_n)_{n\in\N}$
of points in $M$ converges to a point of $S_\infty(M)$ if and only if $r_n\to
\infty$ and $\vartheta_n\to\vartheta\in S^{n-1}$.

Given a continuous function $f\colon S_\infty(M)\to \R$ the \textit{Dirichlet
problem at infinity} is to find a harmonic function $h\colon M\to \R$ which
extends continuously to $S_\infty(M)$ and there coincides with the given 
function $f$, i.e.,
$$h\vert S_\infty(M)=f.$$
The Dirichlet problem at infinity is called \textit{solvable} if
this is possible for every such function $f$. 
In this case a rich class of non-trivial bounded harmonic functions on $M$ 
can be constructed via solutions of the Dirichlet problem at infinity.

In 1983, Anderson \cite{Anderson:83} proved that the Dirichlet problem at infinity is 
indeed uniquely solvable for Cartan-Hadamard manifolds of pinched negative curvature,
i.e.~for complete simply connected Riemannian manifolds $M$ whose sectional
curvatures satisfy
$$-a^2\le\Sect^M_x \le-b^2 \text{ for all $x\in M$,}$$
where $a^2 > b^2 >0$
are arbitrary constants. The proof uses \textit{barrier functions} and 
Perron's classical method to construct harmonic functions. 
Along the same ideas Choi \cite{Choi:84} showed in 1984 that in rotational 
symmetric case of a \textit{model\/} 
$(M,g)$ the Dirichlet problem
at infinity is solvable if the radial curvature is bounded from above by
$-A/(r^2\log(r))$. Hereby a Riemannian manifold $(M,g)$ is called \textit{model\/} if it
possesses a pole $p\in M$ and every linear isometry $\varphi\colon T_pM\to
T_pM$ can be realized as the differential of an isometry $\Phi\colon M\to M$ with
$\Phi(p) = p$. Choi \cite{Choi:84} furthermore provides a
criterion, the \textit{convex conic neighbourhood condition}, which is sufficient for 
solvability of the Dirichlet problem at infinity.

\begin{definition}[cf.~Choi \cite{Choi:84}]
  A Cartan-Hadamard manifold $M$ satisfies the \textit{convex
    conic neighbourhood condition at $x \in S_\infty(M)$} if for any $y\in
  S_\infty(M)$, $y\not = x$, there exist subsets $V_x$ and $V_y\subset M \cup
  S_\infty(M)$ containing $x$ and $y$ respectively, such that $V_x$ and $V_y$ are disjoint open sets of $M\cup
  S_\infty(M)$ in terms of the cone topology and $V_x\cap M$ is convex with
  ${C}^2$-boundary. If this condition is satisfied for all $x\in S_\infty(M)$,
  $M$ is said to satisfy the \textit{convex conic neighbourhood condition}.
\end{definition}

It is shown in Choi \cite{Choi:84} that the Dirichlet problem at infinity is
solvable for a Cartan-Hadamard manifold $M$ with sectional curvature bounded
from above by $-c^2$ for some $c>0$, if $M$ satisfies the convex conic neighbourhood
condition.

In probabilistic terms, if the Dirichlet
problem at infinity for $M$ is solvable and almost surely $$X_\zeta :=
\lim_{t\to\zeta} X_t$$ exists in $S_\infty(M)$, where $(X_t)_{t<\zeta}$ is a
Brownian motion on $M$ with lifetime $\zeta$, the unique solution $h\colon M\to \R$
to the Dirichlet problem at infinity with boundary function $f$ is given as
\begin{equation}
h(x) = \E\left[f\circ X_{\zeta(x)}(x)\right].
\end{equation}
Here $X(x)$ is a Brownian motion starting at $x\in M$.

Conversely, supposing that for Brownian motion $X(x)$ on $M$ almost surely
$$\lim_{t\to \zeta(x)} X_t(x)$$ exists in $S_\infty(M)$ for each $x\in M$, one may 
consider the \textit{harmonic measure} $\mu_x$ on $S_\infty(M)$, where for a
Borel set $U\subset S_\infty(M)$
\begin{equation}
\mu_x(U) := \mathbb{P}\left\{X_{\zeta(x)}(x)\in U\right\}.
\end{equation}
Then, for any Borel set $U\subset S_\infty(M)$, the assignment
$$x\mapsto \mu_x(U)$$
defines a bounded harmonic function $h_U$ on $M$. 
By the maximum principle $h_U$ is either
identically equal to $0$ or $1$ or takes its values in $]0,1[$. Furthermore, all
harmonic measures $\mu_x$ on $S_\infty(M)$ are equivalent. 
Thus, by showing that
the harmonic measure class on $S_\infty(M)$ is non-trivial, 
we can construct non-trivial bounded harmonic functions on $M$.
To show that, for a given
continuous boundary function $f\colon S_\infty(M)\to \R$, 
the harmonic function 
\begin{equation}
\label{uniq_sol_DP}
h(x) = \int_{S_\infty(M)} f(y) \mu_x(dy).
\end{equation}
extends continuously to the boundary $S_\infty(M)$ and takes there the prescribed
boundary values $f$, we have to show that, 
whenever a sequence of points $x_i$ in $M$ converges to $x_\infty\in 
S_\infty(M)$ then the measures $\mu_{x_i}$ converge weakly to the 
Dirac measure at~$x_\infty$. 
In particular, for a 
continuous boundary function $f\vert S_\infty(M)$, 
the unique solution to the
Dirichlet problem at infinity is given by formula \eqref{uniq_sol_DP}.

The first results in this direction have been obtained by Prat \cite{Prat:71, Prat:75} 
between 1971 and 1975. He proved that on a
Cartan-Hadamard manifold with sectional curvature bounded from above
by a negative constant $-k^2$, $k>0$, Brownian motion is transient,
i.e.,~almost surely all paths of the Brownian motion exit from $M$ at the
sphere at infinity \cite{Prat:75}.  If in addition the
sectional curvatures are bounded from below by a constant $-K^2$, $K>k$, he
showed that the angular part $\vartheta(X_t)$ of the Brownian motion almost
surely converges as $t\to \zeta$.  

In 1976, Kifer \cite{Kifer:76} presented a stochastic proof 
that on Cartan-Hadamard manifolds with sectional curvature pinched between two
strictly negative constants and satisfying a certain additional technical condition, 
the Dirichlet problem at infinity can be uniquely solved.
The proof there was given in explicit terms for the two
dimensional case. The case of a Cartan-Hadamard manifold $(M,g)$ with pinched
negative curvature without additional conditions and arbitrary dimension was finally
treated in Kifer~\cite{Kifer:86}.

Independently of Anderson, in 1983, Sullivan \cite{Sullivan:83} 
gave a stochastic proof of
the fact that on a Cartan-Hadamard manifold with pinched negative curvature the
Dirichlet problem at infinity is uniquely solvable. 
The crucial point has been to prove that the harmonic measure
class is non-trivial in this case. 

\begin{thm}[Sullivan \cite{Sullivan:83}]
  The harmonic measure class on $S_\infty(M) = \partial(M\cup S_\infty(M))$ is
  positive on each non-void open set. In fact, if $x_i$ in $M$ converges to
  $x_\infty$ in $S_\infty(M)$, then the Poisson hitting measures $\mu_{x_i}$
  tend weakly to the Dirac mass at~$x_\infty$.
\end{thm}

In the special case of a Riemannian surface $M$ of negative curvature
bounded from above by a negative constant, Kendall \cite{Kendall:84} 
gave a simple stochastic proof
that the Dirichlet problem at infinity is uniquely solvable. 
He thereby used the fact that every geodesic on the Riemannian
surface ``joining'' two different points on the sphere at infinity divides the
surface into two disjoint half-parts. Starting from a point $x$ on $M$, with
non-trivial probability Brownian motion will eventually stay in one of the
two half-parts up to its lifetime. As this is valid for every geodesic and every
starting point $x$, the non-triviality of the harmonic measure class on
$S_\infty(M)$ follows.

Concerning the case of Cartan-Hadamard manifolds of arbitrary dimension
several results have been published how the pinched curvature assumption can
be relaxed such that still the Dirichlet problem at infinity for $M$ is
solvable, e.g.~\cite{Hsu-March:85} and \cite{Hsu:2003}. 

\begin{thm}[Hsu \cite{Hsu:2003}]
  Let $(M,g)$ be a Cartan-Hadamard manifold. The Dirichlet problem at infinity
  for $M$ is solvable if one of the following conditions is satisfied:
\begin{enumerate}
\item There exists a positive constant $a$ and a positive and nonincreasing
  function $h$ with $\int_0^\infty rh(r) \,dr < \infty$ such that
  $$
  - h(r(x))^2e^{2ar(x)} \le \Ric^M_x \text{ and } \Sect^M_x \le -a^2 \text{
    for all $x\in M$. }$$
\item There exist positive constants $r_0$, $\alpha>2$ and $\beta<\alpha-2$
  such that
  $$
  -r(x)^{2\beta}\le \Ric^M_x \text{ and } \Sect^M_x\le
  -\frac{\alpha(\alpha-1)}{r(x)^2} $$
  for all $x\in M$ with $r(x) \ge r_0$.
\end{enumerate}
\end{thm}

It was an open problem for some time whether the existence of a strictly
negative upper bound for the sectional curvature 
could already be a sufficient condition for the solvability of
the Dirichlet problem at infinity as it is true in dimension~2. 
In 1994 however, Ancona \cite{Ancona:94} constructed a Riemannian manifold with sectional
curvatures bounded above by a negative constant such that the Dirichlet
problem at infinity for $M$ is not solvable. 
For this manifold Ancona discussed the
asymptotic behaviour of Brownian motion. 
In particular, he showed that Brownian
motion almost surely exits from $M$ at a single point $\infty_M$ on the sphere
at infinity. 
Ancona did not deal with the question whether $M$ carries  
non-trivial bounded harmonic functions.
 
Borb\'ely \cite{Borbely:98}
gave another example of a Cartan-Hadamard manifold 
with curvature bounded above by a strictly negative constant,
on which the Dirichlet problem at infinity is not solvable. 
Borb\'ely does not discuss Brownian motion on this manifold, but 
he shows using analytic methods that his manifold supports 
non-trivial bounded harmonic functions.

This paper aims to give a detailed analysis of manifolds of this type 
and to answer several questions. 
It turns out that the manifolds of Ancona and Borb\'ely share quite 
similar properties, at least from the probabilistic point of view. 
On both manifolds the angular behaviour of Brownian motion degenerates 
to a single point, as Brownian motion drifts to infinity. In particular, 
the Dirichlet problem at infinity is not solvable. 
Nevertheless both manifolds possess a wealth of 
non-trivial bounded harmonic functions, which come however from 
completely different reasons than in the pinched curvature case.
Since Borb\'ely's manifold is technical easier to handle, we restrict 
our discussion to this case. 
It should however be noted that all essential 
features can also be found in Ancona's example. 

Unlike Borb\'ely who used
methods of partial differential equations to prove
that his manifold provides an example of a non-Liouville manifold for that the
Dirichlet problem at infinity is not solvable, we are interested in a complete
stochastic description of the considered manifold.
In this paper we give a full description of the Poisson boundary 
of Borb\'ely's manifold by characterizing all shift-invariant events for 
Brownian motion. The manifold is of dimension 3, and  
the $\sigma$-algebra of invariant events is generated by two random variables. 
It turns out that, in order to see the full Poisson boundary, 
the attracting point at infinity to which all Brownian paths converge 
needs to be ``unfolded'' into the 
$2$-dimensional space $\R\times S^1$. 

The asymptotic behaviour of Brownian motion on this manifold 
is in sharp contrast to the case of a
Cartan-Hadamard manifold of pinched negative curvature. Recall that in the 
pinched curvature case the angular part
$\vartheta(X)$ of $X$ carries all relevant information and the limiting angle 
$\lim_{t\to\infty}\vartheta(X_t)$ generates the shift-invariant $\sigma$-field of $X$. All
non-trivial information to distinguish Brownian paths for large times is 
given in terms of the angular projection of $X$ onto $S_\infty(M)$, where only the limiting angle matters.
As a consequence, any bounded harmonic function comes from a solution of the 
Dirichlet problem at infinity, which means on the other hand that any bounded harmonic function 
on $M$ has a continuous continuation to the boundary at infinity.

More precisely, denoting by $h(M)$ the Banach space of bounded harmonic
functions on a Cartan-Hadamard manifold $M$, we have the following result of
Anderson.

\begin{thm}[\cite{Anderson:83}]
  Let $(M,g)$ be a Cartan-Hadamard manifold of dimension $d$, whose sectional
  curvatures satisfy $-a^2\le \Sect^M_x\le -b^2$ for all $x\in M$. Then the
  linear mapping
\begin{equation}
\begin{split}
  P\colon L^\infty(S_\infty(M),\mu) &\to h(M), \\
  f &\mapsto P(f), \quad P(f)(x):= \int_{S_\infty(M)}f d\mu_x
\end{split}
\end{equation}
is a norm-nonincreasing isomorphism onto $h(M)$.
\end{thm}

In the situation of Borb\'ely's manifold 
Brownian motion almost surely exits from
the manifold at a single point of the sphere at infinity, independent of its
starting point $x$ at time $0$. 
In particular, all harmonic measures $\mu_x$ are trivial, and  
harmonic functions of the form $P(f)$ are necessarily constant. 
On the other hand, we are going to show that this manifold supports a huge variety of
non-trivial bounded harmonic functions, necessarily without continuation 
to the sphere at infinity. 
It turns out that the richness of harmonic functions is related to two different kinds 
of non-trivial exit sets for the Brownian motion:

1. Non-trivial shift-invariant information is given in terms of the direction 
along which Brownian motion approaches the point at the sphere at infinity. 
Even if there is no contribution from the limiting angle itself, angular sectors 
about the limit point  allow to distinguish Brownian paths
for large times. 

2. As Brownian motion converges to a single point, in appropriate coordinates 
the fluctuative components (martingale parts) of Brownian motion for large times
are relatively small compared to its drift components. 
Roughly speaking, this means that for large times Brownian motion follows the 
integral curves of the deterministic dynamical system which is given by neglecting 
the martingale components in the describing stochastic differential equations. 
We show that this idea can be made precise by constructing a deterministic vector field 
on the manifold $M$ such that the corresponding flow induces a foliation of $M$ and
such that Brownian motion exits $M$ asymptotically along the leaves of this foliation. 

To determine the Poisson boundary the main problem is then to show that 
there are no other non-trivial exit sets.
This turns out to be the most difficult point and will be done by purely 
probabilistic arguments using time reversal of Brownian motion. 
To this end, the time-reversed Brownian motion starting on the exit boundary 
and running backwards in time is investigated.

It is interesting to note that on the constructed manifold, despite of diverging curvature, 
the harmonic measure has a density (Poisson kernel) 
with respect to the Lebesgue measure. Recall that in the pinched curvature case 
the harmonic measure may well be singular with respect to the surface 
measure on the sphere at infinity 
(see \cite{Kifer:76}, \cite{Ledrappier:87}, \cite{Katok:88}, 
\cite{Cranston-Mueller:87} for results in this direction).
Typically it is the fluctuation of the geometry at infinity which prevents harmonic 
measure from being absolutely continuous. 
Pinched curvature alone does in general not allow to control the angular derivative 
of the Riemannian metric, when written in polar coordinates.

The paper  is organized as follows. 
In Section \ref{Section2} we give the construction of our manifold 
$M$ which up to minor technical modifications is the same as in Borb\'ely~\cite{Borbely:98}: 
we define $M$ as the warped product
$$M:= (H\cup L)\times_g S^1,$$
where $L$ is a unit-speed geodesic in the
hyperbolic space $\H^2$ of constant sectional curvature $-1$ and $H$ is one
component of $\H^2 \setminus L$. The Riemannian metric $\gamma$ on $M$ is the
warped product metric of the hyperbolic metric on $H$ coupled with the
(induced) Euclidean metric on $S^1$ via the function $g\colon H\cup L \to \R_+$,
$$ds_M^2 = ds_{\H^2}^2 + g\cdot ds_{S^1}^2.$$
By identifying points $(\ell,\alpha_1)$ and $(\ell,\alpha_2)$ with $\ell \in
L$ and $\alpha_1,\alpha_2 \in S^1$ and choosing the metric ``near'' $L$ equal to
the hyperbolic metric of the three dimensional hyperbolic space $\H^3$ the
manifold $M$ becomes complete, simply connected and rotationally symmetric
with respect to the axis $L$. We specify conditions the function $g$ has to
satisfy in order to provide a complete Riemannian metric on $M$ for which the sectional
curvatures are bounded from above by a negative constant, and such that the
Dirichlet problem at infinity is not solvable.  
As the construction of the function $g$ is described in detail
in \cite{Borbely:98} we mainly sketch the ideas 
and refer to Borb\'ely for detailed proofs. We explain which properties of
$g$ influence the asymptotic behaviour of the Brownian~paths. 

The probabilistic consideration of the manifold $M$ starts in Section~\ref{section3}. 
We specify the defining stochastic differential equations
for the Brownian motion $X$ on $M$, where we use the component processes $R$,
$S$ and $A$ of $X$ with respect to the global coordinate system
$\{(r,s,\alpha)|\, r\in \R_+,\ s\in \R,\ \alpha\in [0,2\pi)\}$ for $M$.
The non-solvability of the Dirichlet problem at infinity is then 
an immediate consequence of the asymptotic behaviour of the Brownian motion 
(see Theorem \ref{BMconverges}).

It is obvious that the asymptotic behaviour of the Brownian motion
on $M$ is the same as in the case of the manifold of Ancona. 
In particular, it turns out that the
component $R$ of the Brownian motion $X$ almost surely tends to infinity as
$t\to \zeta$. It is a remarkable fact that in contrast to the manifold of
Ancona, where the Brownian motion almost surely has infinite lifetime, we can
show that on $M$ the lifetime $\zeta$ of the Brownian motion is
almost surely finite.

In Section \ref{nontrivialshiftinvariant} we start with the construction of non-trivial 
shift-invariant events. 
To this end, we consider a time change of the Brownian motion such that the
drift of the component process $R$ of $X$ equals $t$, i.e.~such that the time
changed component $\widetilde R$ behaves then comparable to the deterministic curve
$\R_+ \to \R_+, \,t\mapsto r_0 + t$. 
We further show that for a certain function $q$, the process
$${\widetilde Z}_t := \widetilde S_t - \int_0^{\widetilde R_t}q(r)\,dr $$
almost surely converges in $\R$ when $t\to \widetilde \zeta$.  

It turns out that the non-trivial
shift-invariant random variable $A_{\zeta}:= \lim_{t\to\zeta}A_t$ can be
interpreted as one-dimensional angle which indicates from which direction the
projection of the Brownian path onto the sphere at infinity attains the point~$L(\infty)$. 
The shift-invariant random variable
${\widetilde Z}_{\widetilde\zeta}$ indicates along which surface of rotation $C_{s_0}\times
S^1$ inside of $M$ the Brownian paths eventually exit the manifold~$M$. Thereby
$C_{s_0}$ is the trajectory starting in $(0,s_0)\in\R_+\times\R$ of the vector
field $$
V := \frac{\partial}{\partial r} + q(r) \frac{\partial}{\partial s}.
$$

Section \ref{Section5} finally gives a description of the Poisson boundary. 
The main work in this part is to verify 
that there are no other invariant events than the specified ones. 
This is achieved by means of arguments relying on the time-inversed process. 

As a consequence a complete description of all non-trivial bounded 
harmonic functions on the manifold $M$ is obtained. 
Concerning ``boundary Harnack inequality'' the class of 
non-constant bounded harmonic functions shares an amazing feature:
on any neighbourhood in $M$
of the distinguished point at the horizon at infinity, 
a non-trivial bounded 
harmonic function on $M$ attains each value lying strictly 
between the global minimum and the global maximum of the function.

Finally it is worth noting that our manifold also provides an example where 
the $\sigma$-field of terminal events is strictly bigger than the 
$\sigma$-field of shift-invariant events.

\section{Construction of a CH manifold with a sink of curvature at infinity}\label{Section2}
\setcounter{equation}0

Let $L \subset \H^2$ be a fixed unit speed geodesic in the hyperbolic half
plane $$\H^2 = \{(x,y)\in\R^2\mid y>0\}$$
equipped with the hyperbolic metric
$ds^2_{\H^2}$ of constant curvature $-1$.  For our purposes one can assume
without loss of generality $L:=\{(0,y)|y>0\}$ to be the positive $y$-axis.
Let $H$ denote one component of $\H^2\setminus L$ and define a Riemannian
manifold $M$ as the warped product:
\begin{displaymath}
  M:=(H\cup L)\times_g S^1,
\end{displaymath}
with Riemannian metric $$ds_M^2= ds^2_{\H^2} + g\cdot ds^2_{S^1},$$
where
$g:H\cup L\to \R_+$ is a positive ${C}^\infty$-function to be determined
later. By identifying points $(\ell,\alpha_1)$ and $(\ell,\alpha_2)$ with
$\ell\in L$ and $\alpha_1,\alpha_2\in S^1$, we make $M$ a simply connected
space.

We introduce a system of coordinates $(r,s,\alpha)$ on $M$, where for a point
$p\in M$ the coordinate $r$ is the hyperbolic distance between $p$ and the
geodesic $L$, i.e.~the hyperbolic length of the perpendicular on $L$ through
$p$. The coordinate $s$ is the parameter on the geodesic
$\{L(s):\,s\in{]{-\infty},\infty[}\,\}$, i.e.~the length of the geodesic segment
on $L$ joining $L(0)$ and the orthogonal projection $L(s)$ of $p$ onto $L$.
Furthermore, $\alpha \in [0,2\pi[$ is the parameter on $S^1$ when using the
parameterization $e^{i\alpha}$. We sometimes take $\alpha \in \R$, in
particular when considering components of the Brownian motion, thinking of
$\R$ as the universal covering of $S^1$.

In the coordinates $(r,s,\alpha)$ the Riemannian metric on $M\setminus L$
takes the form
\begin{equation}\label{RiemMetric}
  ds_M^2= dr^2+h(r)ds^2+g(r,s)d\alpha^2
\end{equation}
where $h(r) = \cosh^2(r)$.
\begin{figure}
  \includegraphics[height=5.5cm, width=10cm]{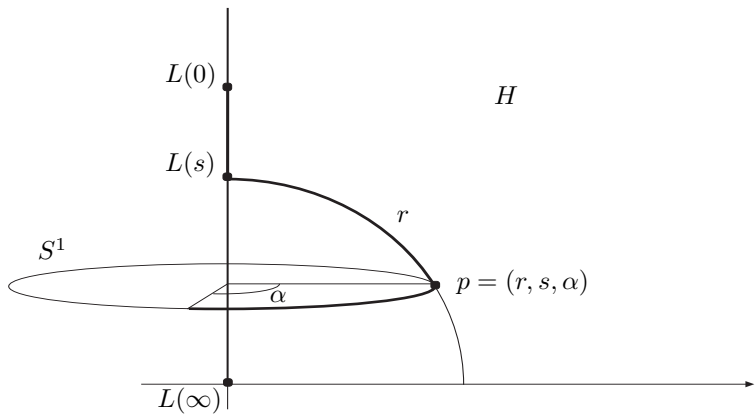}
\caption{Coordinates for the Riemannian manifold $M$}
\label{figure1}
\end{figure}

Let $g(r,s):=\sinh^2(r)$ for $r<{1}/{10}$ (the complete definition is
given below), then the above metric smoothly extends to the whole manifold
$M$, where $M$ is now rotationally symmetric with respect to the axis $L$ and
for $r<{1}/{10}$ isometric to the three dimensional hyperbolic space
$\H^3$ with constant sectional curvature $-1$, cf.~\cite{Borbely:98}. From
that it is clear that the Riemannian manifold $(M,g)$ is complete.

\subsection{Computation of the sectional curvature}\label{subsection21}
From now on we fix the basis $$\p_1:=\frac{\p}{\p r},\  
\p_2:= \frac{\p}{\p s},\ \p_3:=\frac{\p}{\p \alpha}$$ for the tangent space $T_pM$ in $p\in M$.
Herein the Christoffel symbols of the Levi-Civita connection can be computed
as follows -- the indices refer to the corresponding tangent vectors of the
basis:
\begin{alignat*}{3}
  \Gamma^1_{22}&=-\frac12 h'_r,&\qquad \Gamma^2_{12}&=\Gamma^2_{21}=
  \frac{h'_r}{2h},&\qquad
  \Gamma^3_{13}&=\Gamma^3_{31}=\frac{g'_r}{2g}, \\
  \Gamma^1_{33}&=-\frac12 g'_r,&\qquad \Gamma^2_{33}&=-\frac{g'_s}{2h},&
  \qquad \Gamma^3_{23}&=\Gamma^3_{32}=\frac{g'_s}{2g},
\end{alignat*}
all others equal $0$. 
Herein $g'_r$ denotes the partial derivative of the function $g$ with 
respect to the variable r, $g'_s$ the partial derivative with respect to $s$, etc.\\
For the computation of the sectional curvatures $\Sect^M$ write
$X=(x_1,x_2,x_3)$ and $Y=(y_1,y_2,y_3)$ for tangent vectors $X,Y\in T_pM$ in
terms of the basis $\p_1,\p_2,\p_3$. Then one gets
\begin{equation*}
\begin{split}
  \langle R(X,Y)Y,X\rangle &=
  \underbrace{\left(-\frac12h''_{rr}+\frac14\frac{(h'_r)^2}{h}\right)}_{=:A}(x_1y_2-x_2y_1)^2 \\
  &+
  \underbrace{\left(-\frac12g''_{rr}+\frac14\frac{(g'_r)^2}{g}\right)}_{=:B}(x_1y_3-x_3y_1)^2 \\
  &+
  \underbrace{\left(-\frac12g''_{ss}-\frac14g'_rh'_r+\frac14\frac{(g'_s)^2}{g}\right)}_{=:C}(x_2y_3-x_3y_2)^2\\
  &+2\cdot\underbrace{\left(-\frac12g''_{rs}+\frac14\frac{g'_sh'_r}{h}+\frac14\frac{g'_rg'_s}{g}\right)}_{=:D}
  (x_1y_3-x_3y_1)(x_2y_3-x_3y_2).
\end{split}
\end{equation*} 
as well as
\begin{equation*}
\| X\land Y\|^2=
h(x_1y_2-x_2y_1)^2+g(x_1y_3-x_3y_1)^2+gh(x_2y_3-x_3y_2)^2. 
\end{equation*}
We conclude that the manifold $M$ has strictly negative
sectional curvature, i.e. $$-k^2\geq\Sect^M(\text{span}\{X,Y\})= \frac{\langle
  R(X,Y)Y,X\rangle}{\| X\land Y\|^2}$$ for some $k>0$, all $X,Y\in T_pM$ and all
$p\in M$, if and only if the following inequalities hold:
\begin{align}\label{eq:11}
  \frac12h''_{rr}-\frac14\frac{(h'_r)^2}{h}&\geq k^2h,\\ 
  \label{eq:12} \!\!\!\frac12g''_{rr}-\frac14\frac{(g'_r)^2}{g}&\geq
  k^2g,\\ \label{eq:13}
  \frac12g''_{ss}+\frac14g'_rh'_r-\frac14\frac{(g'_s)^2}{g}&\geq
  k^2gh,\\ \notag
\frac{1}{g^2h}\left(-\frac12g''_{rs}
    +\frac14\frac{g'_sh'_r}{h}+\frac14\frac{g'_sg'_r}{g}\right)^2&\leq\left(
    \frac{g''_{rr}}{2g}-\frac{(g'_r)^2}{4g^2}-k^2\right)\\
  &\qquad\times\left(\frac{g''_{ss}}{2gh}+\frac{g'_rh'_r}{4gh}-
    \frac{(g'_s)^2}{4g^2h}-k^2\right).\label{eq:14}
\end{align}
This is explained by the fact that the quadratic form
\begin{displaymath}
  q(X,Y,Z):=(A+k^2h)X^2+(B+k^2g)Y^2+(C+k^2gh)Z^2 + 2DYZ
\end{displaymath}
is non-positive for all $X,Y,Z\in\R$ if and only if
\begin{displaymath}
-A\geq k^2h \,\,\, \text{and} \,\,\, -B\geq k^2g\,\,\, \text{and}\,\,\,-C\geq k^2gh
\,\,\,\text{and}\,\,\, D^2\leq (B+k^2g)(C+k^2gh).
\end{displaymath}

\subsection{The sphere at infinity $S_\infty(M)$}\label{section22}

As it is obvious that, for $(s,\alpha) \in \R\times{[\,0,2\pi[}$ fixed, the
curves
\begin{displaymath}
\gamma_{s\alpha}\colon \R_+\to M,\quad r\mapsto  (r,s,\alpha)
\end{displaymath}
form a foliation of $M$ by geodesic rays, we can describe the sphere at
infinity $S_\infty(M)$ as the union of the ``endpoints'' (i.e.~equivalence
classes) $\gamma_{s\alpha}(+\infty)$ of all geodesic rays $\gamma_{s\alpha}$
foliating $M$ together with the equivalence classes $L(+\infty)$ and
$L(-\infty)$ determined by the geodesic axis $L$ of $M$, which sums up to:
\begin{displaymath}
  S_\infty(M)=L(+\infty)\cup \big\{\gamma_{s\alpha}(+\infty)\,|\,(s,\alpha) \in
  \R\times{[\,0,2\pi[}\big\} \cup L(-\infty).
\end{displaymath}

\begin{figure}
  \includegraphics[height=5.5cm, width=8cm]{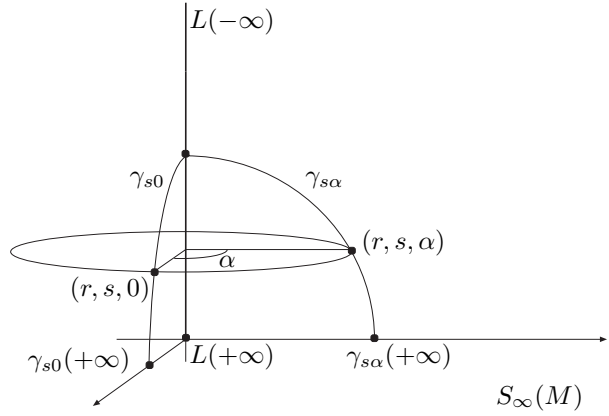}
\caption{Sphere at Infinity $S_\infty(M)$}
\label{figure2}
\end{figure}

This also explains why it suffices to show that the $s$-component $S_t$ of the
Brownian motion $X_t$ converges to $+\infty$ as $t$ approaches the 
lifetime $\zeta$ of the Brownian motion, if we want to show that 
$X_t$ converges for $t \to \zeta$ to the single point
$L(+\infty)$ on the sphere $S_\infty(M)$ at infinity.

\subsection{Properties of the function $g$}\label{section23}

We give a brief description of the properties which the warped product
function $g\colon \R_+\times \R \to \R_+$ in \eqref{RiemMetric}
has to satisfy to provide an example of a Riemannian manifold where the
Dirichlet problem at infinity is not solvable whereas there exist non-trivial
bounded harmonic functions. We give the idea how to construct such a
function $g$  and refer to Borb{\'e}ly \cite{Borbely:98} for further details.

\begin{lemma}\label{propofg}
  There is a ${C}^\infty$-function $g\colon \R_+\times \R \to \R_+$, $(r,s) \mapsto
  g(r,s)$ satisfying the following properties:
 \begin{enumerate}
 \item\label{prop1} $g'_r\geq 0$ and $g'_s\geq 0$.
 \item\label{prop2} $g(r,s) = \sinh^2(r)$ for $r<1/{10}$, and for $r\geq
   1/{10}$ it holds that:
   \begin{displaymath}
     g'_r\geq h'_r g \;\;\text{ and }\;\; \frac12g''_{rr} -
     \frac14\frac{(g'_r)^2}{g} \geq \frac{h'_rg'_r}{8}.
   \end{displaymath}
 \item\label{prop3} Denoting $p(r,s):={g'_s}/({g'_rh})$ one has $(ph)'_s
   \geq 0$ and for all $s\in \R$:
   \begin{displaymath}
     \int_0^\infty p(r,s) \,dr = \infty. 
   \end{displaymath} 
 \item\label{prop4} The function $p(r,s)$ satisfies:
  \begin{displaymath}
    p\leq \frac 1{1000},\ \,p'_s\leq \frac1{1000}, \ |\,p'_r\,|\leq \frac5{1000}
    \ \text{ and }\ pp'_rh^2<\frac{h'_r}{1000}.
  \end{displaymath}
\item\label{prop5} We further need that 
$$
(ph)'_r\geq 0, \quad (ph)''_{rr} \geq 0
$$
and that there exists $\e\in{]0,1/4[}$ such that 
  \begin{equation}
\label{E56}
    (ph)''_{rr}(r) \geq \f{\e}{p(r,0)h(r)}\ \ \text{for all}\ \ r\ge 2.
  \end{equation}
\item\label{prop6} The function $p(r,s)$ finally has to fulfill:
  \begin{displaymath}
     \int_{r_0}^\infty\frac1{p(r,s)h(r)}\,dr = \infty
  \end{displaymath}
  for every $r_0>0$ and every $s\in\R$.
 \end{enumerate}
\end{lemma}

The following remark summerizes briefly the significance of these conditions.

\begin{remark}[Comments on the properties listed above]\hfill\hfill
\begin{enumerate}
\item[1.] As mentioned in the definition of the Riemannian metric on $M$,
  property~(\ref{prop2}) assures that the Riemannian metric smoothly extends to
  the geodesic $L$ and that curvature condition (\ref{eq:12}) is satisfied
  for a suitable~$k$.
  
\item[2.] The conditions in (\ref{prop4}) stated for $p(r,s)$ assure validity of the
  curvature condition (\ref{eq:14}) -- at least for $k=1/{1000}$.
  
\item[3.] The integral condition \eqref{prop3} forces Brownian motion
  $X=(R,S,A)$ on $M$ to converge to the single point
  $L(\infty)\in S_\infty(M)$ as $t\to \zeta$ which as an immediate consequence implies 
  non-solvability of the Dirichlet problem at infinity for $M$. The given
  condition assures that the drift term in the stochastic differential
  equation for the component $S$ of the Brownian motion $X$ compared with the
  drift term appearing in the defining equation for the component $R$ 
  grows ``fast enough'' to ensure $S_t$ going to $\infty$ as
  $t$ approaches the lifetime of the Brownian motion. 

  In \cite{Borbely:98}, Lemma 2.1, Borb{\'e}ly uses this condition to show
  that the convex hull of any neighbourhood of $L(\infty)$ is the whole
  manifold $M$. This is a natural first step in the construction of a manifold
  for which the Dirichlet problem at infinity is not solvable (cf.~the
  Introduction and \cite{Choi:84}, \cite{Borbely:98}).
  
\item[4.] Property (\ref{prop6}) is needed for the Brownian motion to almost 
surely enter a region where $S_t$ has nonzero and  positive drift 
(see the construction of $\ell$ and $\SD$ below); 
otherwise  $S_t$ might converge in $\R$ with a positive probability.
\end{enumerate}
\end{remark}

\subsection{Construction of the function $g$}\label{section24}

The idea to construct a function $g$ with the wanted properties is as follows.
As $g$ is given as solution of the partial differential equation
\begin{equation}\label{pdeforg}
  g'_s(r,s) = p(r,s) h(r) g'_r(r,s) 
\end{equation}
one has to find an appropriate function $p(r,s)$ and the required initial
conditions for $g$ to obtain the desired function.

We will see later that the construction of the metric on $M$ is 
similar to that given in Ancona \cite{Ancona:94}, as
Borb{\'e}ly also defines the function $p$ ``stripe-wise'' to control the
requirements for $p$, $g$ respectively, on each region of the form $[r_i,r_j]
\times \R$. Yet he is mainly concerned with the definition of the ``drift
ratio'' $p$, what makes it harder to track the behaviour of
the metric function $g$ and to possibly modify his construction for other
situations, whereas Ancona gives a more or less direct way to construct the
coupling function in the warped product. As a consequence this allows 
to extend his example to higher dimensions and to adapt it to other
situations as well.

We start with a brief description how to construct the function
\begin{displaymath}
p_0(r):= p(r,0),
\end{displaymath} as given in \cite{Borbely:98}:
$p_0$ is defined inductively on intervals $[r_n,r_{n+1}]$, where
$r_{n+1}-r_n>3$ and $r_1>3$ sufficiently large, see below. We let $p_0 := 0 $
on $[0,2]$ and define it  on $[2,3]$ as a slowly increasing function satisfying 
conditions (\ref{prop4}) and (\ref{prop5}). For $r\in [3,r_1]$ let
$p_0(r):=p_0(3)$ be constant, where $r_1$ is chosen big enough such that
$(p_0h)(r_1) >1$ and $r_1/h(r_1) <1/{1000}$.  On the interval $[r_1,\infty[$
we choose the function $p_0$ to be decreasing with $\lim_{r\to \infty} p_0(r)
= 0$, whereas $p_0h$ is still increasing and strictly convex; see
\cite{Borbely:98}, Lemma 2.3 and Lemma 2.4.

On the interval $[r_1,r_2]$ for $r_2$ big enough as given below (and in
general on intervals of the form $[r_{2n-1},r_{2n}]$) one extends $p_0h$ via a
solution of the differential equation
\begin{displaymath}
  y''= {1}/{(2y)}.
\end{displaymath}
Note that carefully smoothing the function $p_0$ on the interval $[r_1,r_1+1]$
(and on $[r_{2n-1},r_{2n-1}+1]$ respectively) to become ${C}^\infty$ does not interfere
with the properties (\ref{prop4}) and (\ref{prop5}) of $p_0$ and can be done such
that $(p_0h)''>1/{(4p_0h)}$ is still valid.

Lemma 2.4 in \cite{Borbely:98} shows that in fact $p_0\equiv{p_0h}/{h}$ decreases on
$[r_{2n-1},r_{2n}]$ and again by \cite{Borbely:98}, Lemma 2.3,
for given $r_{2n-1}$ one can choose the upper interval bound $r_{2n}$ such
that
\begin{displaymath}
\int_{r_{2n-1}}^{r_{2n}}\frac{1}{p_0h}\, dr>1 \quad\text{ for all } n,
\end{displaymath}
which guarantees property (\ref{prop6}) for $p_0$.

On the interval $[r_2,r_3]$ with $r_3$ sufficiently large as given below and in general
on intervals of the form $[r_{2n}, r_{2n+1}]$, let $p_0={c_n}/{r}$ 
for some well chosen constant $c_n$.  
This differs from the construction of
\cite{Borbely:98} p.~228 where $p_0$ is chosen to be constant in these
intervals, but it does not change the properties of the manifold. As above,
smoothing on intervals $[r_{2n}-1, r_{2n}]$ preserves the conditions
(\ref{prop4}), (\ref{prop5}) and $(p_0h)''>\varepsilon/{(p_0h)}$ for
$0<\varepsilon<1/4$ small enough and independent of $n$, but depending on the
choice of $p_0$ on ${[2,3[}$.

If one chooses for given $r_2$, $r_{2n}$ respectively, the interval bounds
$r_3$, $r_{2n+1}$ respectively, large enough we can achieve that
\begin{displaymath}
  \int_{r_{2n}}^{r_{2n+1}}p_0(r) \,dr >1 \quad\text{ for all } n,
\end{displaymath}
which finally assures property (\ref{prop3}) for $p_0$.

Following Borb{\'e}ly \cite{Borbely:98}, one defines $p(r,s)$ via $p_0(r)$ as
\begin{displaymath}
  p(r,s):=\chi(r,s)p_0(r)
\end{displaymath}
by using a ``cut-off function'' $\chi(r,s)$. Herein $\chi(r,s)$ is given as
$\chi(r,s):=\xi(s+\ell(r))$,
where $\xi$ is smooth and increasing with $\xi(y)= 0$ for $y < 0$, $\xi(y) =
1/2$ for $y>4$ and $\xi',|\xi''|<1/2$, $\xi''+\xi > 0$. The function $\ell$
satisfies $\ell(0)=0$, $\ell'(r)=0$ on the interval $[0,2]$ and
$\ell'(r)={\e}/(p_0h)(r)$ on the interval $[3,\infty[$, with the same $\e$ as
in~\eqref{E56}. Then the two pieces are connected smoothly such that for $r>0$
\begin{equation}
\label{E53}
\ell''(r)\ge -\e\f{(p_0h)'_r(r)}{p_0h(r)}\quad\hbox{and}\quad 0\le \ell'\le\f{\e}{p_0h}.
\end{equation}
The function $\ell$ is nondecreasing such that  $\lim_{r\to\infty}\ell(r) =\infty$
(see \eqref{prop6} in Lemma \ref{propofg}), and one can check that $p(r,s)$
satisfies the required properties listed in Lemma \ref{propofg} (see
\cite{Borbely:98}, p.\,228ff).

For an appropriate initial condition to solve the partial differential
equation (\ref{pdeforg}), we set $\widetilde{g_0}(r):=\sinh^2(r)$ on
$[0,\frac1{10}]$. On $[\frac1{10},\infty[$ let
$\widetilde{g_0}(r)$ be the solution of the differential equation
$$f'= \frac1{\sinh^2\left(1/{10}\right)}\,h'_r f \,\,\text{ with initial
  condition } f\left(\frac1{10}\right) = \sinh^2\left(\frac1{10}\right).$$
Smoothing $\widetilde{g_0}$ yields a ${C}^\infty$-function $g_0$ such that
$\widetilde{g_0}= g_0$ on
${[0,\frac1{10}-\delta]}\cup{[\frac1{10}+\delta,\infty[}$
for an appropriate $\delta$. 
In particular, there exist $d_1,\,d_2>0$ such that
\begin{equation}
\label{d1d2}
g_0(r)=d_2e^{d_1\sinh^2 r}\quad \hbox{for all}\quad r\ge \f1{10}+\d.
\end{equation}
The function $g_0$ serves as initial condition to solve the partial
differential equation
$$
g'_s(r,s) = p(r,s) h(r) g'_r(r,s). $$
\begin{figure}
\includegraphics[width=10cm]{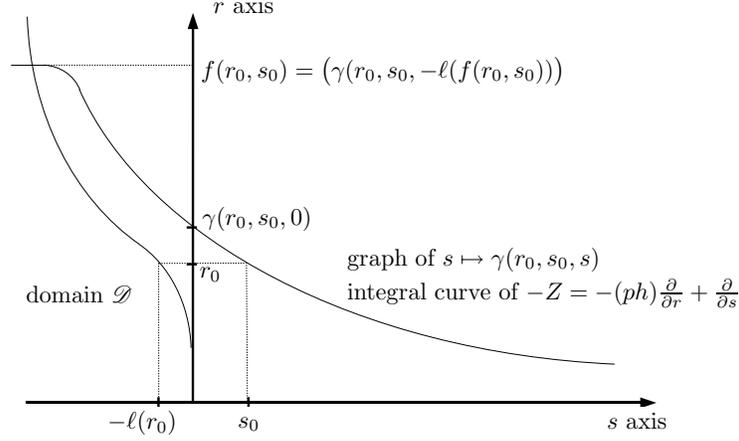}
\caption{Graph of $s\mapsto\gamma(r_0,s_0,s)$}
\label{figure3}
\end{figure}

More precisely, it is proven in \cite{Borbely:98} that all trajectories of the
vector field $-Z:=-(ph)\f{\p}{\p r}+\f{\p}{\p s}$ intersect the domain
$$
\SD:=\left\{(s,r)\in \R\times \R_+,\  s<-\ell(r)\right\}.
$$
Letting $s\mapsto (\g(r_0,s_0,s),s)$ be the integral curve of $-Z$
satisfying $\g(r_0,s_0,s_0)=r_0$, we define $f(r_0,s_0)$ by
\begin{equation}
\label{E54}
f(r_0,s_0)=\g(r_0,s_0,-\ell(f(r_0,s_0))).
\end{equation}
Then (see Sect.~\ref{Section5} for details)
\begin{equation}
\label{E55}
g(r_0,s_0)=g_0(f(r_0,s_0)).
\end{equation}

\section{Asymptotic behaviour of Brownian motion on $M$}\label{section3}

Let $(\Omega;\SF;\P)$ be a filtered probability space satisfying the usual
conditions and $X$ a Brownian motion on $M$ considered as a diffusion process
with generator $\frac12\Delta_M$ taking values in the Alexandroff
compactification $\widetilde M:= M\cup \{\infty\}$ of $M$.  Further let $\zeta$
denote the lifetime of $X$, i.e.~$X_t(\omega) = \infty$ for $t\geq
\zeta(\omega)$, if $\zeta(\omega)<\infty$.

As we have fixed the coordinate system $$M = \big\{(r,s,\alpha): r\in \R_+,\ s\in
\R,\ \alpha \in {[0,2\pi[}\big\}$$ for our manifold $M$, we consider the Brownian
motion $X$ in the chosen coordinates as well and denote by $(R_t)_{t<\zeta}$,
$(S_t)_{t<\zeta}$ and $(A_t)_{t<\zeta}$ the component processes of
$(X_t)_{t<\zeta}$.
The generator $\frac12\Delta_M$ of $X$ is then written in terms of the basis
$\frac{\p}{\p r},\frac{\p}{\p s},\frac{\p} {\p \alpha} $ of $TM$ as:
\begin{equation}
  \label{DeltaM} 
  \begin{split}
    \frac12\Delta_M &= \frac12\left(\frac{\p^2}{\p r^2}+\frac1h\frac{\p^2}{\p
        s^2} + \frac 1g\frac{\p^2}{\p \alpha^2}\right) +
    \left(\frac{h'_r}{4h}+
      \frac{g'_r}{4g}\right)\frac{\p}{\p r}+ \frac{g'_s}{4gh}\frac{\p}{\p s}.
  \end{split}
\end{equation}

Therefore we can write down a system of stochastic differential equations for
the components $R, S$ and $A$ of our given Brownian motion:
\begin{alignat}{3}
  \label{SDEforBM}
  dR_t &= \left(\frac{h'_r(R_t)}{4h(R_t)}+
    \frac{g'_r(R_t,S_t)}{4g(R_t,S_t)}\right)& dt\,\, &+ &dW^1\\
  dS_t &= \frac{g'_s(R_t,S_t)}{4g(R_t,S_t)h(R_t)}&dt\,\, &+
  \quad\;\frac{1}{\sqrt{h(R_t)}}&dW^2\\
  dA_t &= \quad& &\quad\;\;\frac1{\sqrt{g(R_t,S_t)}}\,\, &dW^3
\end{alignat}
with a three-dimensional Euclidean Brownian motion $W=(W^1,W^2,W^3)$.

As already mentioned, we are going to read the component $A$ of the Brownian
motion with values in the universal covering $\R$ of $S^1$.

\begin{remark}\label{lifetimeofB}
  Inspecting the defining stochastic differential equations for the
  components $R$, $S$ and $A$ of the Brownian motion $X$ on $M$, it is 
  interesting to note that the behaviour of the component $A_t$ does not
  influence the behaviour of the components $R_t$ and $S_t$. 
  The first two equations can be solved independently of the third one; 
  their solution then defines the third component of the Brownian motion.
  Hence it is clear that the
  lifetime $\zeta$ of $X_t$ does not depend on the component $A_t$ -- in
  particular does not depend on the starting point $A_0$ of $A_t$. We are
  going to use this fact later to prove existence of non-trivial
  bounded harmonic functions on $M$.
 \end{remark}

As we we are going to see in Section \ref{geominterpretation} the ``drift ratio''
 $p(r,s) ={g'_s}/({g'_rh})$ influences the interplay of the components $S_t$
 and $R_t$ of the Brownian motion and therefore determines the behaviour
 of the Brownian paths. For this reason it is more convenient to work with a
 time-changed version $\widetilde{X}_t$ of our Brownian motion, where the
 drift of the component $\widetilde{R}_t$ is just $t$ and the drift of
 $\widetilde{S}_t$ is essentially given by $p$. This can be realized with a
 time change $(\tau_t)$ defined as follows:

 Let $$T(t) := \int_0^t \left(\frac{h'_r}{4h} +
   \frac{g'_r}{4g}\right)(S_u,R_u)\,du $$
 and $\tau_t:= T^{-1}(t) \equiv
 \inf\{s\in\R_+ : T(s) \geq t\}$ for $t\leq T(\zeta)$. The components
 $\widetilde{R}_t,\widetilde{S}_t, \widetilde{A}_t$ of the time-changed
 Brownian motion $\widetilde{X}_t:= X_{\tau_t}$ are then given for $t\leq
 \widetilde{\zeta} :=T(\zeta)$ by the following system of stochastic
 differential equations:
\begin{alignat}{3}
  \label{SDEfortimechangedBM1}
  d\widetilde{R}_t &= & dt\,\, &+& \frac{1}{\sqrt{\frac{h'_r}{4h}(\widetilde
      R_t) +\frac{g'_r}{4g}(\widetilde R_t,\widetilde S_t)}}
  \,\,&dW^1\\ \label{SDEfortimechangedBM2} d\widetilde{S}_t &=
  \frac{g'_s(\widetilde R_t,\widetilde S_t)}{(g h'_r+ g'_rh)(\widetilde
    R_t,\widetilde S_t)}\,\,\,&dt\,\, &+& \frac{1}{\sqrt{h(\widetilde
      R_t)\left(\frac{h'_r}
        {4h}(\widetilde R_t)+\frac{g'_r}{4g}(\widetilde R_t,\widetilde S_t)\right)}}\,\,&dW^2\\
        \label{SDEfortimechangedBM3}
        d\widetilde{A}_t &=&&& \frac1{\sqrt{g(\widetilde R_t,\widetilde
            S_t)\left(\frac{h'_r} {4h}(\widetilde
              R_t)+\frac{g'_r}{4g}(\widetilde R_t,\widetilde S_t)\right)}}\,\,
        &dW^3.
\end{alignat}

We are now going to state the main theorem of this chapter which
shows that from the stochastic point of view the Riemannian manifold
$(M,\gamma)$ constructed by Borb{\'e}ly \cite{Borbely:98} has essentially
the same properties as the manifold of Ancona in \cite{Ancona:94}. We further
give a stochastic construction of non-trivial bounded harmonic functions on
$M$, which is more transparent than the existence proof of Borb{\'e}ly relying
on Perron's principle.

\begin{thm}[Behaviour of Brownian motion on $M$]\label{BMconverges}\hfill
  \begin{enumerate}
  \item\label{I1} For the Brownian motion $X$ on the Riemannian manifold
    $(M,\gamma)$ constructed above the following statement almost surely
    holds:
  \begin{displaymath}
    \lim_{t\to \zeta}X_t = L(+\infty),
  \end{displaymath}
  independently of the starting point $X_0$. In particular the Dirichlet
  problem at infinity for $M$ is not solvable.
\item\label{I2} The component $A_t$ of Brownian motion $X_t$ almost surely
  converges to a random variable $A_\zeta$ which possesses a positive density
  on $S^1$.
\item\label{I3} The lifetime $\zeta$ of the process $X$ is a.s. finite.
  \end{enumerate}
 \end{thm}

 \begin{proof}
From Eqs.~\eqref{SDEfortimechangedBM1}, ~\eqref{SDEfortimechangedBM2} and
   ~\eqref{SDEfortimechangedBM3} we immediately see that the derivatives of the
   brackets of $\widetilde R$ and $\widetilde S$ are bounded, and that the
   drift of $\tilde S$ takes its values in $[0,\max p]$ and is therefore bounded.
   As a consequence the process $(\widetilde R,\widetilde S, \widetilde A)$
   has infinite lifetime.
   
   Moreover for all $\beta>0$, we have almost surely $|R_t-t|\le \beta t$ eventually.
   From this and the fact that ${g_r'}/{g}\ge h_r'$ for $r\ge {1}/{10}$
   we deduce that the martingale parts of $\widetilde R$ and $\widetilde S$
   together with the process $\widetilde A$ converge as $t$ tends to
   infinity.
   
   Next we see from $$\f{g_s'}{gh_r'+g_r'h}=\f{phg_r'}{gh_r'+g_r'h}\le p$$
   that for $t$ sufficiently large the drift of $\widetilde S$ is larger that
   $p_0((1+\beta)t)/2$. Consequently $\widetilde S$ tends to infinity as
   $t\to\infty$.
To prove \eqref{I1} it is sufficient to establish
   $$
   \lim_{t\to \infty}\widetilde X_t = L(+\infty)
   $$
   and this is a consequence of the fact that $\widetilde S_t$ converges to
   infinity.
The proof of~\eqref{I2} is a direct consequence of the convergence of
   $\widetilde A_t$ to a random variable which conditioned to $\widetilde R$
   and $\widetilde S$ is Gaussian (in $\R$) and non-degenerate.
We are left to prove~\eqref{I3}. Changing back the time of the process
   $\widetilde X$, we get
\begin{equation}
\label{CBT}
\zeta=\int_0^\infty\left(\f{h_r'}{4h}+\f{g_r'}{4g}\right)^{-1}\left(\widetilde R_t,\widetilde S_t\right)\,dt.
\end{equation}
But, for large $r$ and $s\ge 0$, we have
$$
\left(\f{h_r'}{4h}+\f{g_r'}{4g}\right)(r,s)\ge\f{g_r'}{4g}\ge \f{1}{2}\cosh
r\sinh r.
$$
As a.s. $\widetilde R_t-t\ge -\beta t$ eventually and as $\widetilde S_t$ converges
to infinity as $t$ tends to infinity, we get the result.
 \end{proof}

\section{Non-trivial shift-invariant events}\label{nontrivialshiftinvariant} 
\setcounter{equation}0

As explained in the Introduction there is a one-to-one correspondence between
the $\sigma$-field ${\SF}_{\inv}$ of shift-invariant events for $X$ up to
equivalence and the set of bounded harmonic functions on $M$. 
We are going to use this fact and give a probabilistic proof for the 
existence of non-trivial shift-invariant random variables, which in turn yields 
non-trivial bounded harmonic functions on $M$. In addition, we get
a stochastic representation of the constructed harmonic functions as 
``solutions of a modified Dirichlet problem at infinity''. 
In contrast to the usual Dirichlet problem at infinity however, the boundary function does not
live on the geometric horizon $S_{\infty}(M) \cong S^{d-1}(M)$, 
and the harmonic functions are not representable in terms of the limiting angle 
of Brownian motion.

\subsection{Shift-invariant variables} \label{shiftinv_var}
In the discussed example it turns out that the shift-invariant random variable
$A_\zeta$ is non-trivial and can be interpreted as ``$1$-dimensional angle''
on the sphere $S_\infty(M)$ at infinity. The variable $A_\zeta$ gives the
direction on the horizon, wherefrom the Brownian motion $X$
converges to the limiting point $L(\infty)$. 
Despite the fact that the limit $X_\zeta\in S_\infty(M)$ itself is trivial, and hence 
the limiting angle $\vartheta(X)_\zeta\in S^{d-1}(M)$ as well, Brownian paths for large times
can be distinguished by their projection onto $S_\infty(M)$. 
Taking into account that in the pinched curvature case the angular part carries 
all  shift-variant information, one might conjecture 
that the random variable $A_\zeta$ already generates the
shift-invariant $\sigma$-field ${\SF}_{\inv}$.
In turn this would imply a stochastic
representation for bounded harmonic functions $h$ on $M$ as
$$
h(x) = \E^x[f(A_\zeta)]
=\E^x\left[\lim_{t\to\zeta}f(\pr_3(X_t))\right]
$$
with $f\colon S^1\to \R$ measurable. However, Borb{\'e}ly \cite{Borbely:98}
describes a way to construct a family of harmonic functions $\psi(r,s)$ 
which are rotationally invariant, i.e.~independent
of $\alpha$, and therefore cannot be of the above form. 
As he uses ``Perron's principle'' for the construction, 
these harmonic functions do not come with an explicit representation. 
Put in the probabilistic framework, we learn from
this that there must be a way to obtain non-trivial shift-invariant
events also in terms of the components $S_t$ and $R_t$ of the Brownian motion. 
Indeed, as will be seen in Theorem~\ref{zisnotconstant}, there exists a non-trivial
shift-invariant random variable of the form
$$\lim_{t\to\zeta} \left(\widetilde S_t-\int_0^{\widetilde
    R_t}q(r)\,dr\right).$$
Herein $\widetilde S$ and $\widetilde R$ are
time-changed versions of $S$ and $R$ and $q\colon \R_+\to \R$ is a function
already constructed by Borb\'ely, whose properties are listed in
Lemma~\ref{propsofq} below. 
This finally leads to additional (to that depending on
the component $\alpha$) harmonic functions via the stochastic representation
$$h(x) = \E^x\left[g\left(\lim_{t\to\zeta}\left(\widetilde
      S_t-\int_0^{\widetilde R_t}q(r) \,dr\right)\right)\right],$$
where
$g\colon \R\to \R$ is a bounded measurable function.

\begin{lemma}\label{propsofq}
  There is a ${C}^\infty$-function $q\colon \R_+\to \R$ with the following
  properties:
\begin{enumerate}
\item \label{propq1} $$q(r) = -\frac{\sinh(r)}{\cosh^2(r)} =
  \left(\frac1{\sqrt{h}}\right)' \text{ for } r\leq T_1.$$
\item \label{propq2} For $r> T_1$ the function $q$ satisfies the inequalities
  $$
  -3|q| < q' < \frac 1{\cosh(r)}, \qquad\qquad
  \left(\frac1{\sqrt{h}}\right)'\leq q \leq \frac{p_0}{2}-\frac{40}{h}.
  $$
\item \label{propq3} There is a $T_2>T_1$ such that $$q(r) =
  \frac{p_0(r)}{2}-\frac{40}{h(r)} \text{ for } r\geq T_2.$$
\end{enumerate}
\end{lemma} 

\begin{proof}
  See Section \ref{constructionofq} below and \cite{Borbely:98}.
\end{proof}

\begin{thm}\label{zisnotconstant}
  Consider $\widetilde Z_t= \widetilde S_t - \int_0^{\widetilde R_t} q(u)\,du$
  as before.  The limit variable $$\widetilde
  Z_\infty:=\lim_{t\to\infty} \widetilde Z_t$$ exists. Moreover the law of the
  random variable $(\widetilde Z_\infty, \widetilde
  A_\infty)=(Z_\zeta,A_\zeta)$ has full support $\R\times S^1$ and is
  absolutely continuous with respect to the Lebesgue measure.
\end{thm}

\begin{proof}
  Let $t(r,s)=\log g(r,s)$ and
  $$
  m(r,s)=\f2{\sqrt{({h_r'}/{h})(r)+t_r'(r,s)}}.
  $$
  Then a straight-forward calculation shows that
\begin{align}
\begin{split}
  d\widetilde Z_t&=\left(-p(\widetilde R_t,\widetilde S_t)\f{h'_r}{4h}(\widetilde R_t)m^2(\widetilde R_t,\widetilde S_t)-\f12q'(\widetilde R_t)m^2(\widetilde R_t,\widetilde S_t)\right)\,dt\\
  &\qquad+(p(\widetilde R_t,\widetilde S_t)-q(\widetilde R_t))\,dt-q(\widetilde
  R_t)m(\widetilde R_t,\widetilde S_t)\,dW_t^1
+\frac{m(\widetilde R_t,\widetilde S_t)}{\sqrt{h(\widetilde
      R_t)}}\,dW_t^2.
\end{split}
\end{align}
When $\widetilde R_t$ is sufficiently large, this simplifies as
\begin{align}
\begin{split}
  d\widetilde Z_t&=\left(-p(\widetilde R_t,\widetilde S_t)\f{h'_r}{4h}(\widetilde R_t)m^2(\widetilde R_t,\widetilde S_t)-\f12q'(\widetilde R_t)m^2(\widetilde R_t,\widetilde S_t)+\f{40}{h(\widetilde R_t)}\right)\,dt\\
  &\qquad-q(\widetilde R_t)m(\widetilde R_t,\widetilde
  S_t)\,dW_t^1+\frac{m(\widetilde R_t,\widetilde S_t)}{\sqrt{h(\widetilde R_t)}}\,dW_t^2.
\end{split}
\end{align}
 For $t$ sufficiently large, we know that 
$|\widetilde R_t-t|\le \beta t$ (for any small $\beta$) and that
$\widetilde S_t$ is positive; furthermore $t_r'$
is larger than $2\cosh r\sinh r$ for $r$ large and $s\ge 0$. This together
with the fact that the functions $p$, $({h_r'}/{4h})$ $q'$, $q$ and $1/{\sqrt h}$ are
bounded allows to conclude that $\widetilde Z$ is converging.

To prove that the law of $\widetilde Z_\infty$ has full support on $\R$, it is
sufficient to establish convergence of the finite variation part of $\widetilde Z$, 
and that its diffusion coefficient is eventually bounded below by
a continuous positive deterministic function and above by $e^{-\beta t}$ for
some $\beta>0$.  Indeed, with these properties it is easy to prove that, taking
any open non empty interval $I$ of $\R$, there is a time $t$ such that after
time $t$, the process $\widetilde Z$ will hit the center of $I$ and then stay in $I$ with
positive probability.

The convergence of the drift has already been established, together with the
upper bound of the diffusion coefficient. To obtain the lower bound, 
we use Eq.~\eqref{E7} of Lemma~\ref{L1} below, along with the
fact that eventually ${t}/2\le \widetilde R_t\le 2t$ and $\widetilde
S_t\le t$, to obtain
$$
t_r'(\widetilde R_t,\widetilde S_t)\le C_2(h\circ f)(2t,t)\,\f{((p_0h)\circ
  f)(2t,t)}{(p_0h)\left(t/2\right)}
$$
for some $C_2>0$, where $f$ is defined at the end of Sect.~\ref{Section2}. This
immediately yields a lower bound for the diffusion coefficient 
$h(\widetilde R_t)^{-1/2}\,m(\widetilde R_t,\widetilde S_t)$ 
of the form
$$
ce^{-2t}\left((h\circ f)(2t,t)\,\f{((p_0h)\circ
    f)(2t,t)}{(p_0h)\left(t/2\right)}\right)^{-1/2}
$$
(for some $c>0$) which is a positive continuous function.

The assertion on the support of the law of $(\widetilde Z_\infty,
\widetilde A_\infty)$ is then a direct consequence of the fact that conditioned to
$(\widetilde R, \widetilde Z)$, the random variable $\widetilde A_\infty$ 
is a non-degenerate Gaussian variable.

We are left to prove that the law of $(\widetilde Z_\infty, \widetilde
A_\infty)$ is absolutely continuous.  Using the conditioning argument for
$\widetilde A_\infty$, it is sufficient to prove that the law of $\widetilde
Z_\infty$ is absolutely continuous.  
To this end we use the estimates in Sect.~\ref{Section5}. 
It is proven there that there exists an increasing
sequence of subsets $B_n$ of $\Om$ satisfying $\bigcup_{n\ge 0}B_n\asequal\Om$
such that on $B_n$, for each $n\ge 0$, the diffusion process
$$
\breve U_u:= \big(\widetilde R_{\tan u}-\tan u, \widetilde Z_{\tan
    u}\big)
$$
has a continuous extension to $u\in [0,\pi/2]$ and that this extension has
bounded coefficients with bounded derivatives up to order~$2$. Consequently,
following \cite{Kunita:90} Theorem~4.6.5, there exists a $C^1$ flow
$\breve\varphi(u_1,u_2,\om)$ which almost surely realizes a diffeomorphism
from ${]{-\tan u_1},\infty[}\times \R$ to ${]{-\tan u_2},\infty[}\times \R$ if
$u_2<\pi/2$ and from ${]{-\tan u_1},\infty[}\times \R$ to its image in $\R\times
\R$ if $u_2=\pi/2$ (for the first assertion one uses the fact that Brownian
motion has full support).

As a consequence, the density $\breve p_{\pi/4}$ of $\breve U_{\pi/4}$ (which
exists and is positive on ${]{-\tan \pi/4,\infty}[}\times \R$) is transported
by the flow to
$$
\breve p_{\pi/4}\circ
\left(\breve\varphi(\pi/4,\pi/2,\om)\right)^{-1}\big|\det
  J\left(\breve\varphi(\pi/4,\pi/2,\om)\right)^{-1}\big|
$$
where $J\left(\breve\varphi(\pi/4,\pi/2,\om)\right)^{-1}$ is
the Jacobian matrix of $\left(\breve\varphi(\pi/4,\pi/2,\om)\right)^{-1}$.
Integrating with respect to $\om$ proves existence of a positive density
$\breve p_{\pi/2}$ for $\breve U_{\pi/2}$:
$$
\breve p_{\pi/2}= \E\left[ \breve p_{\pi/4}\circ
  \left(\breve\varphi(\pi/4,\pi/2,\newdot)\right)^{-1}\big|\det
    J\left(\breve\varphi(\pi/4,\pi/2,\newdot)\right)^{-1}\big|\right].
$$
Finally projecting onto the second coordinate gives the density of
$\widetilde Z_\infty$, which proves that its law is absolutely continuous.
\end{proof}

\subsection{Construction of the function $q$} \label{constructionofq}

 As the explicit construction of $q$ is already done in
\cite{Borbely:98} we just give a short sketch (following Borb{\'e}ly)
how to get a function~$q$ with the required properties:

Let $a\in \R_+$ and  $T_0$ such that $p_0(r)h(r) > 240$ and $\sqrt{h(r)}
> 80$ for $r> T_0$. Let further $T_1> T_0$ such that $p(r,s) =
p_0(r)/2$ for $r\geq T_1$ and $s\geq a-1$. 
For $r\leq T_1$ the function $q$ is defined as 
$$q(r) :=
\left(\frac1{\sqrt h}\right)' = -\frac{\sinh(r)}{\cosh^2(r)}.$$
For $r\geq T_1$ choose a strictly increasing
${C}^\infty$-extension of $-{\sinh(r)}/{\cosh^2(r)}$ such
that $q(r) > 0$ for $r$ large enough and
$$ \left(\frac1{\sqrt h}\right)''< q' < \frac{1}{\sqrt h}.$$
As $p_0(r)/2 -{40}/{h(r)} > 0 $ with 
$$\lim_{r\to\infty}\left(\frac12 p_0(r) -\frac{40}{h(r)}\right) = 0$$ 
(due to  the construction of $p_0$, see Sect.~\ref{Section1}) there is an
  $r> T_1$ with $q(r) = p_0(r)/2-{h(r)/40}$. 
Let 
$$T_2:= \inf\left\{r>T_1: q(r) = \frac12p_0(r) - \frac{40}{h(r)}\right\}.$$
For $r\geq T_2$ set $$q(r):= \frac12 p_0(r) - \frac{40}{h(r)}.$$
The desired function $q$ is then a smoothened version of this function. 

Borb{\'e}ly shows in \cite{Borbely:98}, p.\,232, that the function $q$ 
obtained this way meets indeed the additionally required properties
$$-3|q|\leq q' \leq \frac1{\sqrt h} \quad\text{and}\quad\left(\frac1{\sqrt
    h}\right)\leq q\leq\frac12 p_0 -\frac{40}{h}\,.$$

\subsection{Interpretation of the asymptotic behaviour of Brownian motion}
\label{geominterpretation} 

We conclude this chapter with some explanations how the
behaviour of the Brownian paths on $M$ should be interpreted
geometrically. 

As seen in Sect.~\ref{Section1} and
Sect.~\ref{nontrivialshiftinvariant}, the random variables 
$$\lim_{t\to\zeta} A_t\quad\text{and}\quad \lim_{t\to \zeta}\left( S_t - \int_0^{ R_t} q(r)\,dr\right)$$ 
serve as non-trivial shift-invariant random variables for $X$ and hence yield
non-trivial shift-invariant events for $X$. 

To give a geometric interpretation of the shift-invariant
variable $$\lim_{t\to\infty} \left(\widetilde S_t - \int_0^{\widetilde
  R_t} q(r) \,dr\right),$$ 
we have to investigate again the stochastic differential
equations for $\widetilde S_t$ and $\widetilde R_t$:
\begin{alignat*}{3}
  d\widetilde{R}_t &= & dt\,\, &+&  
  \frac{1}{\sqrt{\frac{h'_r}{4h}
  +\frac{g'_r}{4g}}}
  \,\,&dW^1_t\\ 
  d\widetilde{S}_t &=
  \frac{g'_s}{g  h'_r+  g'_rh}\,\,\,&dt\,\,
   &+&
  \frac{1}{\sqrt{h\left(\frac{h'_r}
  {4h}+\frac{g'_r}{4g}\right)}}\,\,&dW^2_t.
 \end{alignat*}
We have seen in Sect.~\ref{nontrivialshiftinvariant} that the local
martingale parts 
$$M^1_t = \int_0^t  \left({\frac{h'_r}{4h}
  +\frac{g'_r}{4g}}\right)^{-1/2} dW^1 \text{  and } M^2_t = \int_0^t \left({h\left(\frac{h'_r}
  {4h}+\frac{g'_r}{4g}\right)}\right)^{-1/2} dW^2$$ 
of $\widetilde R_t$ and
$\widetilde S_t$ converge almost surely as $t\to\widetilde\zeta$. 
This suggests that the
component $\widetilde R_t$, when observed at times $t$ near 
$\infty$ (or when starting $X$ near $L(\infty)$), should behave 
like the solution $r(t) := r_0 + t$ of the deterministic
differential equation 
$$\dot r  = 1.$$ 
From the stochastic differential equation above we get
$$\widetilde S_t= S_0 + \int_0^t\frac{g'_s(\widetilde R_t,\widetilde
  S_t)}{g(\widetilde R_t,\widetilde S_t)h'_r(\widetilde R_t) +
  h(\widetilde R_t)g'_r(\widetilde R_t,\widetilde S_t)} \,ds + M^2_t,$$
where the local martingale $M^2_t$ converges as $t\to \infty$, and
$\widetilde R_t$ is expected to behave like $r_0 + t$, when the
starting point $(r_0,s_0,\alpha_0)$ of $X$ is close to $L(\infty)$. 
One might therefore expect $\widetilde S_t$ to behave (for $t$ near to $\infty$) 
like the solution $s(t)$, starting in~$s_0$, of the deterministic
differential equation 
\begin{equation}\label{deterministiccurve}
 \dot s = \frac{g'_s(r(t),s)}{g(r(t),s)h'_r(r(t)) + h(r(t))g'_r(r(t),s)}.
\end{equation}
It remains to make rigorous the meaning of ``should behave
like''. 

Considering the solutions $r(t), s(t)$ of the deterministic
 differential equations above, we note that
$\Gamma_{s_0}\colon \R_+\to \R_+\times \R$ given as $\Gamma_{s_0}(t) := (t, s(t))$ with
$\Gamma_{s_0}(0)= (0,s_0)$ is the trajectory of the ``drift'' vector field 
\begin{equation}\label{driftvectorfeld}
 V_d = \frac{\partial}{\partial r} + \frac{g'_s}{gh'_r+ hg'_r} \frac{\partial}{\partial s}
\end{equation}
starting in $(0,s_0) = L(s_0)$. 

As we are going to see below (see Remark \ref{limfortrajectories}), the
``endpoint'' $$\Gamma_{s_0}(+\infty)\equiv \lim_{t\to\infty} \Gamma_{s_0}(t)$$ of all the trajectories
is just $L(\infty)\in S_\infty(M)$. 
Furthermore, for every point $(r,s)\in\R_+\times \R$,
there is exactly one trajectory $\Gamma_{s_0}$ of $V_d$ with $\Gamma_{s_0}(r) =
(r,s)$, in other words, the union 
$$\bigcup_{s_0\in \R} \Gamma_{s_0}$$ 
defines a foliation of $H$. Recall that $H$ is one component of
$\H\setminus L$ and $M = (H\cup L)\times_g S^1$. Defining a coordinate 
transformation
\begin{equation}
\begin{split} 
\Phi\colon  \R_+ \times \R &\to \R_+\times \R\\
(r,s) &\mapsto (r,s_0)\equiv(\Phi_r(r,s),\Phi_s(r,s)),
\end{split}
\end{equation}
where $s_0$ is the starting point of the unique trajectory $\Gamma_{s_0}$
with $\Gamma_{s_0}(r) = (r,s)$, we obtain coordinates for $\R_+\times \R$ where
the trajectories $\Gamma_{s_0}$ of $V_d$ are horizontal lines. 

Applying the coordinate transformation $\Phi$ to the components
$\widetilde R_t$ and $\widetilde S_t$ of the Brownian motion, 
$$\Phi(\widetilde R_t,\widetilde
S_t) = (\Phi_r(\widetilde R_t,\widetilde S_t), \Phi_s(\widetilde
R_t,\widetilde S_t)),$$
we are able to compare the behaviour of the components $\widetilde R_t$ and
$\widetilde S_t$  with the trajectories $\Gamma_{s_0}$ of $V_d$,
in other words, with the deterministic solutions $r(t)$ and $s(t)$. 
The component $\Phi_r(\widetilde R_t,\widetilde S_t)$ obviously equals 
$\widetilde R_t$. Yet, knowing that
for $t\to\infty$ the new component $\Phi_s(\widetilde R_t,\widetilde
S_t)$ possesses a non-trivial limit, would mean that the Brownian paths (their
projection onto $(H\cup L)$, to be precise) finally approach the point $L(\infty)\in
S_\infty(M)$ along a (limiting) trajectory $\Gamma_{s_0}$,
where $s_0 = \lim_{t\to\infty} \Phi_s(\widetilde R_t,\widetilde
S_t)$. 
This would contribute another piece of non-trivial information to the asymptotic behaviour 
of Brownian motion, 
namely along which trajectory (or more precisely, along which surface of
rotation $\Gamma_{s_0}\times S^1$) a Brownian path finally exits the manifold~$M$.

It remains to verify that the above defined
component $\Phi_s(\widetilde R_t,\widetilde S_t)$ indeed has a
non-trivial limit as $t\to \infty$. As already seen,
$\Phi_s(\widetilde R_t,\widetilde S_t)$ is the starting point of the
deterministic curve $s(t)$, satisfying the differential equation
(\ref{deterministiccurve}) with $s(\widetilde R_t) = \widetilde
S_t$. The solution $s(t)$ is of the form 
$$s(t) = s_0 +\int_0^tf(r(u),s(u)) \,du$$ 
where $f= {g'_s}/{(gh'_r + hg'_r)}$. In
particular, $s(t)$ explicitly  depends on $s(u)$ for $u\leq t$. That is
the reason why, when applying It{\^o}'s formula to $\Phi_s(\widetilde
R_t,\widetilde S_t)$, there appear first order derivatives of the flow 
$$\Psi\colon \R_+\times \R \to \R_+\times \R,\quad \quad (r,s) \mapsto \Gamma_s(r)$$
with respect to the variable $s$. 
Estimating these terms does not seem to be trivial, nor
to provide good estimates in order to establish convergence
of $\Phi_s(\widetilde R_t,\widetilde S_t)$ as $t\to\infty$.

A possibility to circumvent this problem is to find a vector
field $V$ on $T(\R_+\times\R)$ of the form $\partial/\partial r + f(r) \partial/\partial s$ whose
trajectories also foliate $H$ and are not ``far off''  the
trajectories $\Gamma_{s_0}$ of $V_d$ -- in particular the trajectories of
$V$ have to  exit $M$ through the point $L(\infty)\in S_\infty(M)$ as well.

As seen in Sect.~\ref{nontrivialshiftinvariant}, we have 
$$\left|\,\frac{g'_s}{gh'_r+hg'_r} - p\,\right| \leq
\left|\,p\cdot\frac{1}{1+h}\,\right|. $$ 
Furthermore, for $r\geq T_2$ the function $q(r)$ is defined as $p_0/2 -
{40}/{h}$, in particular $q(r)$ does not differ much from the function
$p(r,s)$ which equals $p_0/2$ for $r$ and $s$ large. Hence $q(r)$
is  a good approximation for $g_s/(gh'_r + hg'_r)$ for $r$
large, and is independent of the variable $s$.

We therefore consider the vector field 
\begin{equation}\label{importantvectorfield}
V:= \frac{\partial}{\partial r} + q(r) \frac{\partial}{\partial s}. 
\end{equation} 
Starting in $(0,s_0)\in\R_+\times \R$ the trajectories $C_{s_0}$ of $V$ have the form
$$C_{s_0}(t) = \left(t,s_0+\int_0^{t} q(u) \,du\right).$$ 
As we are going to see below,
we also have $\lim_{t\to \infty}C_{s_0}(t) = L(\infty)$, see Remark
\ref{limfortrajectories},  and the union 
$$\bigcup_{s_0\in  \R} C_{s_0}$$ 
forms a foliation of $H$.

For $(r,s)\in\R_+\times\R$ there is exactly one
trajectory $C_{s_0}$ of $V$ with $C_{s_0}(r) = s$. Its starting point
$s_0$ can be computed as $s_0 = s-\int_0^r q(u) \,du$. We can therefore
define a coordinate transformation 
\begin{figure}
\includegraphics[width=12cm]{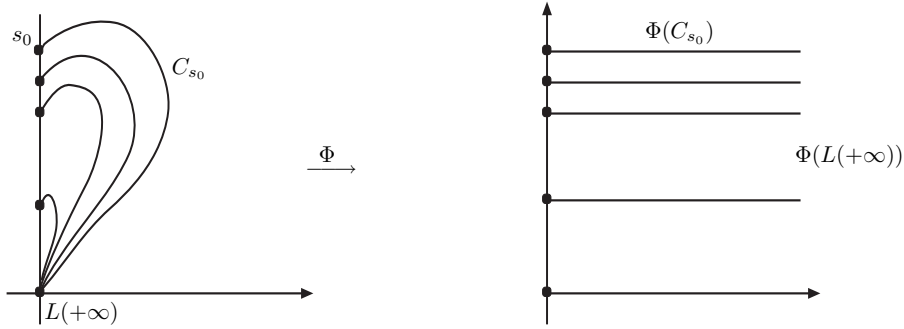}
\caption{Effect of the coordinate transformation $\Phi$}
\label{figure4}
\end{figure}
\begin{equation}
\begin{split}
 \Phi\colon \R_+\times \R &\to \R_+\times \R,\\
(r,s) &\mapsto \left(r,s - \int_0^rq(u)\,du\right).
\end{split}
\end{equation}
As seen in Figure~\ref{figure4}, the trajectories
$C_{s_0}$ of $V$ are horizontal lines in the new coordinate system. 

In the changed coordinate system the components $\widetilde R_t$ and
$\widetilde S_t$ of $X$ then look like 
\begin{equation}
\Phi(\widetilde R_t, \widetilde S_t) = \left(\widetilde R_t,\widetilde S_t -
\int_0^{\widetilde R_t} q(u)\,du\right)\equiv \big(\Phi_r(\widetilde R_t,\widetilde S_t),
\Phi_s(\widetilde R_t,\widetilde S_t)\big).
\end{equation}
As we have proven in Sect.~\ref{nontrivialshiftinvariant}, 
$$\lim_{t\to\widetilde \zeta} \left[\widetilde S_t -\int_0^{\widetilde R_t} q(u)
\,du\right]\equiv \lim_{t\to\widetilde \zeta} \Phi_s(\widetilde R_t,\widetilde S_t)$$
exists and is a non-trivial shift-invariant random variable. Therefore
the non-triviality of $\lim_{t\to\infty} \Phi_s(\widetilde
R_t,\widetilde S_t)$ allows to distinguish 
Brownian paths when examining along which of the trajectories
$C_{s_0}$ of $V$, or more precisely along which surface of rotation
$C_{s_0}\times S^1$, the path eventually exits the manifold. 
This gives the geometric significance of the limit variable
$$\lim_{t\to\infty}\left[\widetilde S_t -\int_0^{\widetilde R_t} q(u)\,du\right].$$

It finally remains to complete the section with the proof that the
trajectories of the vector field $V$, as well as the trajectories of
the vector field $V_d$, exit the manifold $M$ in the point $L(\infty)$. 
This is done in the final remark.

\begin{remark}\label{limfortrajectories}
We have:
$$\lim_{t\to \infty} C_{s_0}(t) = L(\infty) \quad\text{ and }\quad \lim_{t\to\infty}
\Gamma_{s_0}(t) = L(\infty)\quad \text{for every } s_0 \in \R. $$
\end{remark}
 
\begin{proof}
It is enough to show that the ``$s$-component'' of each trajectory
$C_{s_0}$, $\Gamma_{s_0}$ resp., converges to $\infty$ with $t\to \infty$. The
s-component of $C_{s_0}$ is 
$$ s(t) = s_0 + \int_0^tq(r)\,dr,$$
the s-component of $\Gamma_{s_0}$ 
$$ s(t) = s_0 + \int_0^t\frac{g'_s(r,s(r))}{g(r,s(r))h'_r(r) + h(r)g'_r(r,s(r))}\,
dr.$$
Since for $t\geq T_2$ we have $$q(r) = \frac12 p_0(r) - \frac{40}{h},$$ it
follows immediately that $\lim_{t\to\infty} \int_0^tq(r)\,dr = \infty$ because
$\int_0^\infty\frac{40}{h(r)} \,dr <\infty$ and 
$$\lim_{t\to\infty}\int_0^tp_0(r) \,dr = \infty,$$ due to
Lemma \ref{propofg}, property \eqref{prop3}.

For the second term we notice that $s_0 +
\int_0^tf(r,s(r))\,dr $ with $f=g'_s/(gh'_r + hg'_r)$ is
nondecreasing as the integrand is positive. Moreover we have seen above
and in the foregoing sections that 
$$\left|\,\frac{g'_s}{gh'_r +
 hg'_r}-p\,\right|\leq \left|\,p\cdot\frac1{1+h}\,\right|.$$ 
As 
$$\int_0^\infty
p(r,s)\frac1{1+h(r)}\,dr \leq \int_0^\infty\frac1{1+h(r)} \,dr <\infty$$ it suffices
to show that $\lim_{t\to\infty} \int_0^tp(r,s(r))\,dr = \infty$. This is true as
$s(r) \geq s_0$ for all $r\leq t$ and therefore for $r$ large enough we have
$p(r,s(r)) = p_0(r)/2$. The claimed result then follows exactly as above. 
\end{proof}

\section{The Poisson boundary of $M$}\label{Section5}
\setcounter{equation}0

In this section we prove that any shift-invariant event for Brownian motion
in $M$ is measurable with respect to the random variable $(Z_\zeta,A_\zeta)$
constructed in Sect.~\ref{nontrivialshiftinvariant}.  
As a consequence this allows
a complete characterization of the Poisson boundary of~$M$.

We perform the change of variable $\Psi \colon (r,s,a)\mapsto (r,z,a)$ with
$z=s-\int_0^rq(u)\,du$. Further let $t(r,s)=\log g(r,s)$ and
$$
m(r,s)=\f2{\sqrt{\f{h_r'}{h}(r)+t_r'(r,s)}}.
$$
In the new coordinates, the three components ${\widetilde
  X}_t(x)=({\widetilde R}_t(x),{\widetilde Z}_t(x), {\widetilde
  A}_t(x))=({\widetilde R}_t,{\widetilde Z}_t,{\widetilde A}_t)$ 
of time-changed Brownian motion satisfy
\begin{align}
\label{E1}
\begin{split}
  d{\widetilde R}_t&=dt +m({\widetilde R}_t,{\widetilde S}_t)\,dW^1_t,\\
  d{\widetilde Z}_t&=\left(-p({\widetilde R}_t,{\widetilde S}_t)\f{h'_r}{4h}({\widetilde R}_t)m^2({\widetilde R}_t,{\widetilde S}_t)+-\f12q'({\widetilde R}_t)m^2({\widetilde R}_t,{\widetilde S}_t)\right)\,dt\\
  &\quad\ +(p({\widetilde R}_t,{\widetilde S}_t)-q({\widetilde R}_t))\,dt-q({\widetilde R}_t)m({\widetilde R}_t,{\widetilde S}_t)\,dW_t^1+\f{m({\widetilde R}_t,{\widetilde S}_t)}{\sqrt{h({\widetilde R}_t)}}\,dW_t^2,\\
  d{\widetilde A}_t&=\f{m({\widetilde R}_t,{\widetilde S}_t)}{\sqrt{g({\widetilde R}_t,{\widetilde
        S}_t)}}\,dW^3_t.
\end{split}
\end{align}
In case when ${\widetilde R}_t$ exceeds the positive constant $T_2$ defined in
Lemma~\ref{propsofq}, the equations simplify to
\begin{align}
\label{E2}
\begin{split}
  d{\widetilde R}_t&=dt +m({\widetilde R}_t,{\widetilde S}_t)\,dW_t^1,\\
  d{\widetilde Z}_t&=m^2({\widetilde R}_t,{\widetilde S}_t)\left(-\f12p_0({\widetilde R}_t)\f{h'_r}{4h}({\widetilde R}_t)-\f12q'({\widetilde R}_t)\right)\,dt+\f{40}{h(\widetilde R_t)}\,dt\\
  &\qquad+m({\widetilde R}_t,{\widetilde S}_t)\left(-q({\widetilde R}_t)\,dW_t^1+\f{1}{\sqrt{h({\widetilde R}_t)}}\,dW_t^2\right),\\
  d{\widetilde A}_t&=\f{m({\widetilde R}_t,{\widetilde S}_t)}{\sqrt{g({\widetilde R}_t,{\widetilde
        S}_t)}}\,dW^3_t.
\end{split}
\end{align}
In the subsequent estimates we always assume $r\ge T_2$ and $s\ge 0$.
Since $p_0$ is nonincreasing and $p_0h$ is increasing and strictly convex, we
have $(p_0h)'\ge 0$ which yields
$$
0\ge p_0'\ge -p_0\f{h'}{h}
$$
and consequently
\begin{equation}
\label{E3}
|p_0'(r)|\le 2p_0(r),\quad |q'(r)|\le p_0(r).
\end{equation}
It is easy to see that
\begin{equation}
\label{E4}
|(h^{-1/2})'|\le h^{-1/2}.
\end{equation}

We want to estimate $m$, $m_r'$ and $m_s'$.  We know that $t_s'=ph \,t_r'$.
Recall that for $s_0\ge 0$, $r_0>0$, the curve $s\mapsto (\g(r_0,s_0,s),s)$ is the
integral curve to the vector field $\partial_s-ph(r,s)\partial_r$ satisfying
$\g(r_0,s_0,s_0)=r_0$. Recall further that $f(r_0,s_0)$ is defined by
$$
f(r_0,s_0)=\g\left(r_0,s_0,-\ell(f(r_0,s_0))\right),
$$
see~Eq.~\eqref{E54}, and that by definition of the metric $g$, we have
$$
g(r_0,s_0)=g_0\left(f(r_0,s_0)\right)\quad \hbox{where}\ \ 
g_0(r)=d_2e^{d_1\sinh^2 r}
$$
for some $d_1,\, d_2>0$, see~Eq.~\eqref{E55}.  
Letting $$t_0(r)=\log g_0(r),$$
the function $t(r,s)=\log g(r,s)=t_0(f(r,s))$ satisfies
\begin{equation}
\label{E5}
t(r_0,s_0)=d_1\sinh^2 (f(r_0,s_0))+\log d_2,
\end{equation}
which yields
\begin{equation}
\label{E6}
|t_r'|\le 2d_1(h\circ f)|f_r'|\quad \hbox{and}\quad|t_s'|=ph|t_r'|\le 2d_1ph(h\circ f)|f_r'|.
\end{equation}

\begin{lemma}
\label{L1}
There exists two constants $C_1, C_2 >1$ such that on the set $\{(r,s)\in {[3,\infty[}\times
{[0,\infty[}\}$,
\begin{align}
\label{E8}
\f{(p_0h)\circ f}{p_0h}\le f_r'\le C_1 \f{(p_0h)\circ f}{p_0h},
\end{align}
and
\begin{align}
\label{E7}
{C_2}^{-1}h\circ f\f{(p_0h)\circ f}{p_0h}\le t_r'\le C_2\, h\circ
f\f{(p_0h)\circ f}{p_0h}.
\end{align}
\end{lemma}

\begin{proof}
  It suffices to verify \eqref{E8}. Indeed, assuming that \eqref{E8} is true, then from the equality $t_r'=(t_0'\circ f)f_r'$
  and the bound $ h(r)/{C'}\le t_0'(r)\le C' h(r)$ for some $C'>1$ and
  all $r\ge 3$, we obtain \eqref{E7}.
  
  To establish \eqref{E8}, let $(r_0,s_0)\in {[3,\infty[}\times
  {[0,\infty[}$ and denote by $\g_i'$ the derivative of $\g$ with respect to the
  $i$-th variable, $i=1,2,3$. We then have
\begin{equation}
\label{E9}
\g_3'(r_0,s_0,s)=-(ph)\left(\g(r_0,s_0,s),s\right)
\end{equation}
and
\begin{equation}
\label{E10}
\g(r_0,s_0,-\ell(f(r_0,s_0)))=f(r_0,s_0).
\end{equation}
Differentiating \eqref{E10} yields
\begin{align*}
  f_r'&(r_0,s_0)\\&=\g_1'(r_0,s_0,-\ell(f(r_0,s_0)))-\g_3'(r_0,s_0,-\ell(f(r_0,s_0)))\ell'(f(r_0,s_0))f_r'(r_0,s_0)\\
  &=\g_1'(r_0,s_0,-\ell(f(r_0,s_0))+(ph)\left(-\ell(f(r_0,s_0)),f(r_0,s_0)\right)\ell'(f(r_0,s_0))f_r'(r_0,s_0)\\
  &=\g_1'(r_0,s_0,-\ell(f(r_0,s_0)),
\end{align*}
since $ph(r,s)=\chi(r,s)(p_0h)(r)$ and
$\chi\left(f(r_0,s_0),-\ell(f(r_0,s_0))\right)=0$ (see \cite{Borbely:98} p.~229).  
Consequently it is sufficient to prove that
\begin{align}
\label{E12}
\f{(p_0h)\circ f}{p_0h}(r_0,s_0)\le \g_1'(r_0,s_0,-\ell(f(r_0,s_0))\le C
\f{(p_0h)\circ f}{p_0h}(r_0,s_0).
\end{align}

To this end, we differentiate \eqref{E9} with respect to~$r$ and obtain
\begin{equation}
\label{E13}
\g_{13}''(r_0,s_0,s)=-(ph)_r'\left(\g(r_0,s_0,s),s\right)\g_1'(r_0,s_0,s).
\end{equation}
Solving the last equation with initial condition $\g_1'(r_0,s_0,s_0)=1$ (which
comes from $\g(r_0,s_0,s_0)=r_0$) gives
\begin{align*}
  &\g_1'\left(r_0,s_0,-\ell(f(r_0,s_0))\right)\\
  &=\exp\left(\int_{-\ell(f(r_0,s_0))}^{s_0}(ph)_r'(\g(r_0,s_0,s),s)\,ds\right)\\
  &=\exp\left(\int_{-\ell(f(r_0,s_0))}^{s_0}\chi(\g(r_0,s_0,s),s)(p_0h)'(\g(r_0,s_0,s))\,ds\right)\\
  &\quad \quad
  \times\exp\left(\int_{-\ell(f(r_0,s_0))}^{s_0}\chi_r'(\g(r_0,s_0,s),s)(p_0h)(\g(r_0,s_0,s))\,ds\right)\\
  &=\exp\left(\int_{-\ell(f(r_0,s_0))}^{s_0}ph(\g(r_0,s_0,s),s)\f{(p_0h)'(\g(r_0,s_0,s))}{(p_0h)(\g(r_0,s_0,s))}\,ds\right)\\
  &\quad \quad
  \times\exp\left(\int_{-\ell(f(r_0,s_0))}^{s_0}\chi_r'(\g(r_0,s_0,s),s)(p_0h)(\g(r_0,s_0,s))\,ds\right)\\
  &=\exp\left(\int_{-\ell(f(r_0,s_0))}^{s_0}-\g_3(r_0,s_0,s)\f{(p_0h)'(\g(r_0,s_0,s))}{(p_0h)(\g(r_0,s_0,s))}\,ds\right)\\
  &\quad \quad
  \times\exp\left(\int_{-\ell(f(r_0,s_0))}^{s_0}\chi_r'(\g(s_0,r_0,s),s)(p_0h)(\g(r_0,s_0,s))\,ds\right)
\end{align*}
where for the last equality we used \eqref{E9}. 
We thus get
\begin{align*}
  \g_1'&\left(r_0,s_0,-\ell(f(r_0,s_0))\right)\\
  &=\exp\left(\int_{-\ell(f(r_0,s_0))}^{s_0}-\f{d}{ds}\left(\log (p_0h)(\g(r_0,s_0,s))\right)\,ds\right)\\
  &\quad \quad
  \times\exp\left(\int_{-\ell(f(r_0,s_0))}^{s_0}\chi_r'(\g(r_0,s_0,s),s)(p_0h)(\g(r_0,s_0,s))\,ds\right)\\
  &=\f{p_0h\left(f(r_0,s_0)\right)}{p_0h(r_0)}
  \exp\left(\int_{-\ell(f(r_0,s_0))}^{s_0}\chi_r'(\g(r_0,s_0,s),s)p_0h(\g(r_0,s_0,s))\,ds\right)\\
    &=\f{p_0h\left(f(r_0,s_0)\right)}{p_0h(r_0)}\exp\left(\int_{-\ell(f(r_0,s_0))}^{s_0}\xi'(s+\ell(\g(r_0,s_0,s)))
  (\ell'p_0h)(\g(r_0,s_0,s))\,ds\right)\\
    &=\f{p_0h\left(f(r_0,s_0)\right)}{p_0h(r_0)}\exp\left(\int_{-\ell(f(r_0,s_0))}^{s_0}\xi'(s+\ell(\g(r_0,s_0,s)))
  \e\,ds\right),
\end{align*}
since for $s\in [-\ell(f(r_0,s_0)), s_0]$, $\di \g(r_0,s_0,s)\ge
\g(r_0,s_0,s_0)=r_0\ge 3$ and $\ell'(r)={\e}/({p_0h(r)})$ for $r\ge 3$ (see
\cite{Borbely:98} p. 229 and Sect.~\ref{Section2}). Hence we have
\begin{align*}
  &\g_1'\left(r_0,s_0,-\ell(f(r_0,s_0))\right)\\
  &=\f{p_0h\left(f(r_0,s_0)\right)}{p_0h(r_0)}\exp\left(
  \int_{-\ell(f(r_0,s_0))}^{-\ell(f(r_0,s_0))+8}\xi'(s+\ell(\g(r_0,s_0,s)))
  \e\,ds\right)
\end{align*}
where we used the fact that $\xi'(s')=0$ when $s'>4$, and
$$\f{d}{ds}(s+\ell(\g(r_0,s_0,s))) =1-\ell'ph(\g(r_0,s_0,s))\ge 3/4$$
(\cite{Borbely:98} (2.15) and $\e\le 1/4$).  

Since $\xi'$ is nonnegative and
bounded by $1/2$, we have
$$
1\le\exp\left(\int_{-\ell(f(r_0,s_0))}^{-\ell(f(r_0,s_0))+8}\xi'(s+\ell(\g(r_0,s_0,s)))
\e\,ds\right)\le \exp(4\e)
$$
which finally gives
\begin{equation}
\label{E14}
\f{(p_0h)\circ f}{p_0h}(r_0,s_0)\le \g_1'\big(r_0,s_0,-\ell(f(r_0,s_0))\big)
\le \exp(4\e) \f{(p_0h)\circ f}{p_0h}(r_0,s_0).
\end{equation}
This is the desired result.
\end{proof}

\begin{lemma}
\label{L2}
There exist constants $C_1, C_2>0$ such that
\begin{align}
\label{E17}
|f_{rr}''|\le C_1 \left(\f{(p_0h)\circ f}{p_0h}\right)^2
\end{align}
and
\begin{align}
\label{E15}
|t_{rr}''|\le C_2 h\circ f\left(\f{(p_0h)\circ f}{p_0h}\right)^2.
\end{align}
\end{lemma}

\begin{proof}
 Assume that \eqref{E17} is true.  From the identity
\begin{equation}
\label{E16}
t_{rr}''=(t_0''\circ f)(f_r')^2+(t_0'\circ f)f_{rr}'',
\end{equation}
along with~\eqref{E8} and the fact that $0\le t_0'\le 2d_1 h$, $0\le
t_0''\le 4 d_1h$, we obtain~\eqref{E15}.

To establish \eqref{E17} we first note that
\begin{align*}
  f_{r}'(r_0,s_0)&=\g_{1}'(r_0,s_0,-\ell(f(r_0,s_0))\end{align*} and
\begin{align}
  &\g_{1}'(r_0,s_0,-\ell(f(r_0,s_0)))\notag \\
  &\qquad=\f{p_0h\left(f(r_0,s_0)\right)}{p_0h(r_0)}
  \exp\left(\int_{-\ell(f(r_0,s_0))}^{-\ell(f(r_0,s_0))+8}\xi'(s+\ell(\g(r_0,s_0,s)))
  \e\,ds\right).\label{E18}
\end{align}
Hence, we have
\begin{align}\label{formula5.18}
\begin{split}  &f_{rr}''(r_0,s_0)\\
  &=\f{(p_0h)'\left(f(r_0,s_0)\right)f_r'(r_0,s_0)}{p_0h(r_0)}
  \exp\left(\int_{-\ell(f(r_0,s_0))}^{-\ell(f(r_0,s_0))+8}\xi'(s+\ell(\g(r_0,s_0,s)))
  \e\,ds\right)\\
  &\quad-\f{p_0h\left(f(r_0,s_0)\right)(p_0h)'(r_0)}{(p_0h)^2(r_0)}
  \exp\left(\int_{-\ell(f(r_0,s_0))}^{-\ell(f(r_0,s_0))+8}\xi'(s+\ell(\g(r_0,s_0,s)))
  \e\,ds\right)\\
  &\quad+f_{r}'(r_0,s_0)\big(\ell'(f(r_0,s_0))f_r'(r_0,s_0)\xi'(0)\big)\\
   &\quad+f_{r}'(r_0,s_0)\int_{-\ell(f(r_0,s_0))}^{-\ell(f(r_0,s_0))+8}\xi''(s+\ell(\g(r_0,s_0,s)))
    \ell'(\g(r_0,s_0,s))\g_1'(r_0,s_0,s) \e\,ds.
\end{split}
\end{align}
Using $0\le (p_0h)'\le 2p_0h$, along with the boundedness of the integral in the
exponential, the bound for $f_r'$, the boundedness of $\xi'$, $\xi''$ and
$\ell'$ (\cite{Borbely:98} (2.15)), and the fact that
\begin{align*}
  0&\le \g_1'(r_0,s_0,s)\\
  &=\exp\left(\int_s^{s_0}(ph)_r'( \g(r_0,s_0,u),u)\,du\right)\\
  &\le \exp\left(\int_{-\ell(f(r_0,s_0))}^{s_0}(ph)_r'(
    \g(r_0,s_0,u),u)\,du\right)\quad\hbox{(since $(ph)_r'\ge 0$, see
    \cite{Borbely:98} p. 229)}\\
  &=\g_{1}'(r_0,s_0,-\ell(f(r_0,s_0))=f_r'(r_0,s_0),
\end{align*}
we get the wanted bound.
\end{proof}

\begin{lemma}
\label{L3}
There exist $\b>1$ and $r_0>0$ such that for all $r\ge r_0$ and $s\ge 0$,
\begin{equation}
\label{E19}
 f(r,s)\ge (p_0h)^\b(r).
\end{equation}
\end{lemma}
\begin{proof}
  Since $f(r,s)\ge f(r,0)$ for $s\ge 0$,  it is sufficient to establish
  \eqref{E19} for $s=0$.  
Let $r_0$ be such that $\ell(r_0)\ge 4$. For $r\ge r_0$, define $f_1(r)$ as
  $$
  \g(r,0,-\ell(r)+4)=f_1(r).
  $$
  Since $-\ell(f(r,0))<-\ell(r)+4$, we have $f_1(r)<f(r,0)$. Hence it is
  sufficient to establish \eqref{E19} with $f_1(r)$ in place of $f(0,r)$.
  Writing
  $$
  -\ell(r)+4=\int_r^{f_1(r)}\f{d}{du}\g^{-1}(r,0,\cdot)(u)\,du
  $$
  we obtain
  $$
  \ell(r)-4=\int_r^{f_1(r)}\f{1}{ph(u,\g^{-1}(r,0,\cdot)(u))}\,du.
  $$
  But when $u\in [r, f_1(r)]$, we have
  $$\g^{-1}(r,0,\cdot)(u)\ge
  \g^{-1}(r,0,\cdot)(f_1(r)))=-\ell(r)+4,$$
  hence
  $$
  \ell(u)+\g^{-1}(r,0,\cdot)(u)\ge \ell(u)-\ell(r)+4\ge 4,
  $$
and this implies 
  $$
  \chi(u,\g^{-1}(r,0,\cdot)(u))=\xi(\ell(u)+\g^{-1}(r,0,\cdot)(u))=\f12.
  $$
 Thus, since $ph(r,s)=\chi(r,s)p_0(r)h(r)$, we obtain
\begin{equation}
 \label{E50}
\ell(r)-4=\int_r^{f_1(r)}\f{2}{p_0h(u)}\,du.
\end{equation}
Differentiating with respect to $r$ yields
$$
\ell'(r)=\f{2f_1'(r)}{p_0h(f_1(r))}-\f{2}{p_0h(r)},
$$
which combined with $\ell'=\e/(p_0h)$ gives
\begin{equation}
 \label{E51}
f_1'(r)=\f{(2+\e)\,p_0\,h(f_1(r))}{2p_0h(r)}.
\end{equation}
From the convexity of $p_0h$ it follows
$$
p_0h(f_1(r))\ge (p_0h)'(r)\,(f_1(r)-r).
$$
By Eq.~\eqref{E50}, taking into account that $p_0h$ is nondecreasing, we derive 
the inequality
$f_1(r)-r\ge (\ell(r)-4)\,p_0h(r)$. Choosing $\l\in {]0,1[}$
such that $\b:=\l\,({2+\e})/{2}>1$, this implies $f_1(r)-r>\l
f_1(r)$ for $r$ sufficiently large. We get
$$
p_0h(f_1(r))\ge (p_0h)'(r) \l f_1(r)
$$
and with \eqref{E51}
\begin{equation}
 \label{E52}
\f{f_1'(r)}{f_1(r)}\ge \b \f{(p_0h)'(r)}{p_0h(r)}.
\end{equation}
Integrating \eqref{E52} from $r_0$ sufficiently large and noting that
$f_1(r_0)\ge p_0h(r_0)$, we finally arrive at the desired result.
\end{proof}

We are now in a position to estimate $ m'_r$ and $m'_s$.  Since $h'_r/h$ and
its derivatives are bounded, we can estimate $|m'_r|$ by 
$C\left|\partial_r(t_r')^{-1/2}\right|$. Recall that $t(r,s)=t_0(f(r,s))$ which
gives
$$
(t_r')^{-1/2}=(t_0'\circ f)^{-1/2}(f'_r)^{-1/2}
$$
and
$$
\partial_r(t_r')^{-1/2} =-\f12(t_0'\circ f)^{-3/2}(t_0''\circ
f)(f_r')^{1/2}-\f12(t_0'\circ f)^{-1/2}(f_r')^{-3/2}(f_{rr}'').
$$
The last equation, along with \eqref{E8} and \eqref{E17}, and $(1/C)h\le
t_0,\ t_0',\ t_0''\le C h$ for some $C>1$ and all $r$ sufficiently large
(recall that $t_0(r)=d_1\sinh^2r+\log d_2$), gives
\begin{equation}
\label{E20}
|m_r'|\le C(p_0\circ f)^{1/2}(p_0h)^{-1/2}.
\end{equation}
Alternatively, using Eq.~\eqref{E20} together with Lemma \ref{L1} and estimating
$m^{-1}$ by $(t_r')^{1/2}$, we obtain
\begin{equation}
\label{E22}
\f{|m_r'|}{m}\le C\f{(p_0h)\circ f}{(p_0h)}.
\end{equation}

Similarly to $m_r'$, we estimate $|m_s'|$ with $
C\,|\partial_s(t_r')^{-1/2}|$. Since $t_s'=(ph)t_r'$, we have
\begin{align*}
  \partial_s(t_r')^{-1/2}&=-\f12 (t_r')^{-3/2}t_{sr}''\\
  &=-\f12 (t_r')^{-3/2}\partial_r\left(\f12p_0ht_r'\right)\\
  &=-\f12
  (t_r')^{-3/2}\left(\partial_r\left(\f12p_0h\right)t_r'+\f12p_0ht_{rr}''\right)\\
  &=-\f12 (t_r')^{-1/2}\partial_r\left(\f12p_0h\right)-\f14
  p_0h(t_r')^{-3/2}t_{rr''}.
\end{align*}
Recall that $|(p_0h)'|\le 2p_0h$.  The first term on the right can be bounded in
absolute value by
$$
C(h\circ f)^{-1/2}((p_0h)\circ f)^{-1/2}(p_0h)^{3/2}
$$
which is smaller than $C(p_0\circ f)^{1/2}(p_0h)^{1/2}$ since
$$
\f{p_0h}{h\circ f}\le \f{p_0h}{h\circ (p_0h)^\b}\le 1.
$$
By means of Lemma~\ref{L1} and Lemma~\ref{L2}, the second term on the
right can be bounded in absolute value by
$$
C(p_0\circ f)^{1/2}(p_0h)^{1/2}.
$$
Consequently we obtain
\begin{equation}
\label{E23}
\big|\partial_s(t_r')^{-1/2}\big|\le C(p_0\circ f)^{1/2}(p_0h)^{1/2}.
\end{equation}
The following lemma is a consequence of Lemma \ref{L3}, and Eqs.~\eqref{E20}, \eqref{E23}.
\begin{lemma}
\label{L4}
There exists $r_1\ge 0$ such that for every $r\ge r_1$ and $s\ge 0$,
\begin{equation}
\label{E24}
|m_r'|\le C(p_0\circ f)^{1/2}\,(p_0h)^{-1/2},\qquad  |m_s'|\le C(p_0\circ f)^{1/2}\,(p_0h)^{1/2},
\end{equation}
\begin{equation}
\label{E26}
\f{|m_r'|}{m}\le C\,\f{(p_0h)\circ f}{(p_0h)}\quad 
 \hbox{and}\quad  \f{|m_s'|}{m}\le C\big((p_0h)\circ
  f\big).
\end{equation}
\end{lemma}

We recall that
\begin{equation}
\label{E27}
C^{-1}(h\circ f)^{-1/2}\f{(p_0h)^{1/2}}{((p_0h)\circ f)^{1/2}}\le m\le 
C(h\circ f)^{-1/2}\f{(p_0h)^{1/2}}{((p_0h)\circ f)^{1/2}}
\end{equation}
for some $C>1$.
Let us briefly explain how bounds for higher order derivatives can be
established. Differentiating expression \eqref{formula5.18}
with respect to $r$ yields, for $s\ge 0$ and $r$ sufficiently large, 
\begin{equation}
\label{Efrrr}
|f_{rrr}^{(3)}|\le C\left(\f{(p_0h)\circ
  f}{p_0h}\right)^3.
\end{equation}
(We remark that differentiation of each term in \eqref{formula5.18} amounts to multiplication by
an expression smaller than $C({(p_0h)\circ f})/({p_0h})$ in
absolute value). This easily yields, for $s\ge 0$ and $r$ sufficiently large,
\begin{equation}
\label{E28}
|t_{rrr}^{(3)}|\le C(h\circ f)\left(\f{(p_0h)\circ
  f}{p_0h}\right)^3
\end{equation}
and
\begin{equation}
\label{E31}
|m_{rr}''|\le C(h\circ f)^{-1/2}\left(\f{(p_0h)\circ
  f}{p_0h}\right)^{3/2}.
\end{equation}
Now exploiting the fact that $t_s'=pht_r'$, a straightforward 
calculation shows 
(for $s\ge 0$ and $r$ sufficiently large):
\begin{align}
\label{E29}
&|t_{srr}^{(3)}|\le C(p_0h)(h\circ f)\left(\f{(p_0h)\circ
    f}{p_0h}\right)^3,\\
\label{E30}
&|t_{ssr}^{(3)}|\le C(p_0h)^2(h\circ f)\left(\f{(p_0h)\circ
    f}{p_0h}\right)^3,\\
\label{E32}
&|m_{sr}''|\le C(p_0h)(h\circ f)^{-1/2}\left(\f{(p_0h)\circ
    f}{p_0h}\right)^{3/2},\\
\label{E33}
&|m_{ss}''|\le C(p_0h)^2(h\circ f)^{-1/2}\left(\f{(p_0h)\circ
    f}{p_0h}\right)^{3/2},\\
\label{E34}
&|m_{rr}''m|\le C(p_0h)^{-1}(p_0\circ f),\\
\label{E35}
&|m_{sr}''m|\le C(p_0\circ f),\\
\label{E36}
&|m_{ss}''m|\le C(p_0h)(p_0\circ f).
\end{align}

\begin{prop}
\label{P1}
For any $\a\in{\big]0,\f{\b-1}{2(\b+1)}\big[}$ there exists a constant 
$r_\a>0$ such that for all $r\ge r_\a$ and $s\ge 0$, 
\begin{equation}
\label{E41}
|m|,\ |m_r'|,\ |m_s'|,\ |mm_{rr}''|,\
 |mm_{rs}''|,\ |mm_{ss}''|\le h^{-\a}(r).
\end{equation}
\end{prop}

\begin{proof}
  All the estimates are either immediate or easy consequences of the ones for
  $|m_s'|$ and $|mm_{ss}''|$.  On the other hand, since the proof is the same for
  these two functions, we only establish the estimate for $|mm_{ss}''|$.
  Note that the final $\a$ should lie in $]0,\f{\b-1}{2(\b+1)}[$
  where $\b>1$ is given by Lemma~\ref{L3}, due to the estimate $|m_s'|\le
  (p_0h)^{1/2}(p_0\circ f)^{1/2}$.
  
  Fix $\a\in{]0,\f{\b-1}{\b+1}[}$. It is sufficient to prove that
  $h^{\a}(r) (p_0h)(r)(p_0\circ f)(r,s)$ is bounded for $r\ge r_0$ and $s\ge
  0$. Then the claimed result is obtained by picking a smaller~$\a$.
  
  From Lemma~\ref{L3} and $p_0\le 1/r$, we have
  $$
  h^{\a}(r) (p_0h)(r)(p_0\circ f)(r,s)\le h^{\a}(r)(p_0h)^{1-\b}(r).
  $$
  If $h^{\a}(r)(p_0h)^{1-\b}(r)\le 1$, we are done. Otherwise we have
  $(p_0h)(r)\le h^{\a/(\b-1)}(r)$ which yields
  $$
  p_0(r)\le h^{{\a}/({\b-1})-1}(r).
  $$
  In this case, since $f(r,s)\ge r$ and $p_0$ is nonincreasing,
  $$
  h^{\a}(r) (p_0h)(r)(p_0\circ f)(r,s)\le h^{\a}(r)
  h^{{\a}/({\b-1}})(r)h^{{\a}/({\b-1})-1}(r).
  $$
  We finally get
  $$
  h^{\a}(r) (p_0h)(r)(p_0\circ f)(r,s)\le 1\vee
  \left(h^{[\a(\b+1)-(\b-1)]/(\b-1)}(r)\right);
  $$
  since $\a\in{]0,\f{\b-1}{\b+1}[}$ the right hand side is
  clearly bounded.
\end{proof}

\subsection{Equation for the time-reversed process}\label{SectionA3}

Consider  
$${\widetilde X}_t={\widetilde X}_t(x)=({\widetilde R}_t,{\widetilde Z}_t,{\widetilde A}_t)$$ 
the time-changed Brownian motion
started at $x$, satisfying Eq.~\eqref{E1}.

Let $t=\tan u$ and define $\breve X_u=(\breve R_u,\breve
Z_u, \breve A_u)={\widetilde X}_{\tan u}$, $0\le u<\pi/2$. 
Furthermore denote 
\begin{align*}
\bar m(r,z)&=m\circ\Psi^{-1}(r,z,a)=m(r,s),\\
\bar p(r,z)&=p\circ\Psi^{-1}(r,z,a)=p(r,s),\\ 
\bar g(r,z)&=g\circ\Psi^{-1}(r,z,a)=g(r,s). 
\end{align*}
 From \eqref{E1} we get
\begin{align}
\label{E37}
\begin{split}
d\breve R_u&=\f1{\cos^2 u}du +\f1{\cos u}\bar m(\breve
R_u,\breve Z_u)\,d\breve W_u^1,\\
d\breve Z_u&=\f1{\cos^2 u}\,\left(-\breve  p(\breve
R_u,\breve Z_u)\f{h_r'}{4h}(\breve R_u)\bar m^2(\breve R_u,\breve Z_u)-\f12q'(\breve
R_u)\bar m^2(\breve R_u,\breve Z_u)\right)\,du\\
&\quad+\f1{\cos^2 u}\,(\bar p(\breve R_u,\breve Z_u)-q(\breve R_u))\,du\\
&\quad-\f1{\cos u}\,q(\breve R_u)\bar m(\breve R_u,\breve Z_u)\,d\breve
W_u^1+\f1{\cos u}\,\f{\bar m(\breve
R_u,\breve Z_u)}{\sqrt{h(\breve R_u)}}\,d\breve W_u^2,\\
d\breve A_u&=\f1{\cos u}\,\f{\bar m(\breve R_u,\breve
Z_u)}{\sqrt{\bar g(\breve R_u,\breve Z_u)}}\,d\breve W^3_u
\end{split}
\end{align}
for some three-dimensional Brownian motion $(\breve W^1,\breve
W^2,\breve W^3)$. Letting $\breve Y_u=\breve R_u-\tan u$ and $\hat
X_u=(\breve Y_u,\breve Z_u,\breve A_u)=(\hat X_u^1,\hat X_u^2,\hat X_u^3)$, we can write 
\begin{equation}
\label{E38}
d\hat X_u^j=\s_i^j(u,\hat X_u)\,d\breve W_u^i+b^j(u,\hat
X_u)\,du,\quad j=1,2,3,
\end{equation}
where $b^1=b^3=0$, 
\begin{align*}
b^2(u,x)=&-\f1{\cos^2 u}\,\bar p(x^1+\tan
u,x^2)\f{h_r'}{4h}(x^1+\tan u)\,\bar m^2(x^1+\tan u,x^2)\\
&-\f1{\cos^2 u}\,\f12\,q'(x^1+\tan u)\,\bar m^2(x^1+\tan u,x^2)\\
& +\f1{\cos^2 u}\,(\bar p(x^1+\tan u,x^2)-q(x^1+\tan u))
\end{align*}
(the last line equals $\di \f{40}{\cos^2 u \ h(x^1+\tan u)}$ for $ x^1+\tan u$ large),
\begin{align*}
&\s_1^1(u,x)=\f1{\cos u}\,\bar m(x^1+\tan u,x^2),\quad \s_2^1=\s_3^1=0,\\
&\s_1^2(u,x)=-\f1{\cos u}\,q(x^1+\tan u)\,\bar m(x^1+\tan u,x^2),\\
&\s_2^2(u,x)=\f1{\cos u}\,\f{\bar m(x^1+\tan u,x^2)}{\sqrt{h(x^1+\tan u)}},\\
&\s_3^2=\s_1^3=\s_2^3=0,\\
&\s_3^3(u,x)=\f1{\cos u}\,\f{\bar
m(x^1+\tan u,x^2)}{\sqrt{\bar g(x^1+\tan u,x^2)}}.
\end{align*}
Adopting Proposition~\ref{P1} and defining  $\s(\pi/2,x)=0$ and $b(\pi/2,x)=0$, 
we obtain easily a $C^2$  extension
of $\s$ and $b$ on $D= \{(u,x)\in[0,\pi/2]\times \R^3,\ x^1>-\tan u\}$
with vanishing derivatives at $u=\pi/2$.

Let $K$ be a compact subset of $D$. Further let $\s^K$, $b^K$ be two $C^2$ 
maps defined on $[0,\pi/2]\times
\R^3$ which are bounded  together with their  derivatives bounded up to order~$2$, 
and which coincide on $K$ with $\s$ and $b$. 
We fix $\a\in {]0,\pi/2[}$
        and let $\hat X^K$ be defined by  
\begin{equation}
\label{E42}
d\hat X_u^K=\s^K(u,\hat X_u^K)\,d\breve W_u+ b^K(u,\hat X_u^K)\,du\quad \text{with }\hat X_\a^K=\hat X_\a.
\end{equation}
For the diffusion $( \hat X_u^K)_{u\in[\a,\pi/2]}$
conditions (H1), (H2) of \cite{Pardoux:86} are satisfied, and we may 
conclude with corollaire~2.4 of \cite{Pardoux:86}  that the time-reversed process
$$( \check X_u^K=\hat X_{\pi/2-u}^K)_{u\in[0,\pi/2-\a]}$$
satisfies 
\begin{equation}
\label{E43}
d\check X_u^K=\check\s^K(u,\check X_u^K)\,d\check W_u+\check
b^K(u,\check X_u^K)\, du\quad\text{with }\check X_0^K=\hat X_{\pi/2}^K,
\end{equation}
where $\check W$ is a Brownian motion independent of $\check X_0^K$ 
and where the coefficients are given by
\begin{align*}
\check\s^K(u,x)&=\s^K(\pi/2-u,x),\\
(\check b^K(u,x))^j&=-(
 b^K(\pi/2-u,x))^j+(p^K(\pi/2-u,x))^{-1}\f{\partial}{\partial
 x^i}(( a^K)^{ij}p^K)(\pi/2-u,x); 
\end{align*}
here we have  $(a^K)^{ij}=\sum_{k=1}^3(\s^K)_k^i(\s^K)_k^j$ and $p^K(u,x)$ is
the density of $\hat X_u^K$.

\subsection{Invariant events}\label{SectionA4}

\begin{prop}
\label{P2}
For every $x=(r,z,a)\in M$ with $r>0$, up to negligible sets, the pre-image of
$\SF_{\rm inv}$ under the path map $\Phi$ to the Brownian motion
${\widetilde X}=({\widetilde R},{\widetilde Z},{\widetilde A})$, starting from $x$,
is contained in the $\sigma$-algebra $\s(\hat X_{\pi/2})$ 
generated by $\hat X_{\pi/2}$.
\end{prop}

\begin{proof}
  Fix $\a\in {]0,\pi/2[}$.  Since $\SF_{\rm inv}\subset \SF_\infty$, the
  $\sigma$-algebra $\Phi^{-1}(\SF_{\rm inv})$ is clearly included in
  $\di\bigcap_{t>0}\s({\widetilde X}_s,\ s\ge t)=\Phi^{-1}(\SF_{\infty})$.  
On the other hand, we have
\begin{align*}
  \bigcap_{t>0}\s({\widetilde X}_s,\ s\ge
  t)&=\bigcap_{u\in]\a,\pi/2[}\s(\breve X_v,\ 
  v\ge u)\\
  &=\bigcap_{u\in]\a,\pi/2[}\s(\hat X_v,\ 
  v\ge u)\quad \hbox{up to negligible sets}\\
  &=\bigcap_{u\in]0,\pi/2-\a[}\s(\check X_v,\ v\le u).
\end{align*}
For $n\ge 1$, consider the compact subset $K_n$ of $D$,
$$
K_n=\left\{(u,x)\in[0,\pi/2]\times \R^3,\ x^1\ge -\tan u+\f1n,\ |x^i|\le
  n,\ i=1,2,3\right\}
$$
and let
$$
B_n=\big\{\om,\ (u,\hat X_u(\om))\in K_n, \ \forall\, u\in [\a,\pi/2]\big\}.
$$
Since ${\widetilde X}$ has a.s.~continuous paths with positive first
component, we have for any $n\ge 1$
\begin{equation}
\label{E48}
\bigcup_{m\ge n}B_m\asequal\Om\,.
\end{equation}
Now define $\hat X^n=\hat X^{K_n}$ and $\check X^n=\check X^{K_n}$.
For $$A\in \bigcap_{u\in{]0,\pi/2-\a[}}\s(\check X_v,\ v\le u),$$ we find from
Eq.~\eqref{E48}
$$
A\asequal\bigcup_{n\ge 1}A\cap B_n.
$$
Denoting by $\hat\Phi$ (resp.~$\hat\Phi_n$) the path map corresponding to
$\hat X$ (resp. $\hat X^n$) and by $\hat\SF_{\pi/2}$ the $\s$-field of terminal
events for continuous paths $[\a,\pi/2]\to D$, there exists $\hat C\in
\hat\SF_{\pi/2}$ such that
$$
A=\hat\Phi^{-1}(\hat C).
$$
Then on $B_n$, we have $\hat X\asequal \hat X^n$, and hence
\begin{equation}
\label{E46}
A\cap B_n\asequal\hat\Phi_n^{-1}(\hat C)\cap B_n
\end{equation}
where
$$
\hat\Phi_n^{-1}(\hat C)\in \bigcap_{u\in]\a,\pi/2]}\s(\hat X_v^n,\ v\ge u).
$$
On the other hand, since $\check X^n$ is the strong solution of a
stochastic differential equation driven by a Brownian motion $\check W^n$
independent of $\check X_0^n$ (see \cite{Pardoux:86} corollaire~2.4), we have
\begin{align*}
  \bigcap_{u\in]\a,\pi/2[}\s(\hat X_v^n,\ v\ge
  u)&=\bigcap_{u\in]0,\pi/2-\a[}\s(\check X_v^n,\ 
  v\le u)\\
  &=\bigcap_{u\in]0,\pi/2[}\left(\s(\check X_{0}^n)\vee \s(\check W_v^n,\ 
    v\le u)\right)\\
  &=\s(\check X_{0}^n)\vee \bigcap_{u\in]0,\pi/2[}\s(\check W_v^n,\ 
  v\le u)\quad \hbox{(up to negligible sets)}\\
  &=\s(\check X_{0}^n) =\s(\hat X_{\pi/2}^n).
\end{align*}
As a consequence, there exists a Borel subset $E_n$ of $\R^3$ such that
$$
\hat\Phi_n^{-1}(\hat C)=\{\hat X_{\pi/2}^n\in E_n\}.
$$
This yields, along with~\eqref{E46},
\begin{equation}
\label{E47}
A\cap B_n\asequal\{\hat X_{\pi/2}^n\in E_n\}\cap B_n
\asequal\{\hat X_{\pi/2}\in E_n\}\cap B_n\subset \{\hat X_{\pi/2}\in E_n\}.
\end{equation}
The inclusion $A\cap B_n \subset \{\hat X_{\pi/2}\in E_n\}$ together with
\eqref{E48} yields for all $n\ge 1$,
$$
A\subset \left\{\hat X_{\pi/2}\in\bigcup_{m\ge n} E_m\right\}\quad
\hbox{a.s.}
$$
This is true for all $n\ge 1$, so letting $\di E=\limsup_{n\to\infty}E_n$,
we get
$$
A \subset \{\hat X_{\pi/2}\in E\}\quad \hbox{a.s.}
$$
Now let us establish the other inclusion. If $\om\in \{\hat X_{\pi/2}\in E\}$
then choose $n$ such that $(u, \hat X_u(\om))\in K_n$ for all $u\in
[\a,\pi/2]$. Then choose $m\ge n$ so that $\hat X_{\pi/2}(\om)\in E_m$. We
clearly have $\om\in B_m$, so that $\om\in A\cap B_m$ by~\eqref{E47}. Finally we
proved that up to a negligible set,
$$
A=\{\hat X_{\pi/2}\in E\}.
$$
This shows that
$$
\bigcap_{t>0}\s({\widetilde X}_s,\ s\ge t)\subset \s(\hat{X}_{\pi/2})
$$
up to negligible sets.
\end{proof}

Let ${\widetilde Y}_t={\widetilde R}_t-t$ and $\di {\widetilde Y}_\infty
=\lim_{t\to\infty}{\widetilde Y}_t$. We proved that the invariant paths for
the process ${\widetilde X}_t=({\widetilde R}_t,{\widetilde Z}_t,{\widetilde
  A}_t)$ are measurable with respect to $({\widetilde Y}_\infty, {\widetilde
  Z}_\infty,{\widetilde A}_\infty)$.  The last step consists in
eliminating the first variable.
\begin{thm}
\label{T1}
Let $x_0=(r_0,z_0,a_0)\in M$.  The invariant paths for the process
${\widetilde X}_t(x_0)=({\widetilde R}_t(x_0),{\widetilde
  Z}_t(x_0),{\widetilde A}_t(x_0))$ are measurable with respect to
$({\widetilde Z}_\infty(x_0),{\widetilde A}_\infty(x_0))$.
\end{thm}
\begin{proof}
  Let $x_0=(r_0,z_0,a_0)\in M$ and $S_{x_0}$ the sphere of radius~$1$ centered
  at~$x_0$.  Define a time-changed Brownian motion ${\widetilde X}(x_0)$ (time-changed 
  in the sense of a solution to~\eqref{E1}) on a product probability space as
  follows: ${\widetilde X}'(x_0)(\om_1)$ is a time-changed Brownian motion
  started at $x_0$, and for any $x\in S_{x_0}$, ${\widetilde X}''(x)(\om_2)$
  is a time-changed Brownian motion started at $x$. 
Letting $\tau=\inf\{t>0,\ {\widetilde X}'_t(x_0)\in S_{x_0}\}$, we define 
\begin{align*}
{\widetilde X}_t(x_0)(\om_1,\om_2)=
\begin{cases}
{\widetilde X}'_t(\om_1) &\text{if $t\le \tau(\om_1)$,}\\ 
{\widetilde X}''_{t-\tau}({\widetilde
    X}'_\tau(\om_1))(\om_2) &\text{if $t>\tau(\om_1)$.}
\end{cases}
\end{align*}
  
Given $B\in\SF_{\rm inv}$, we need to prove that there exists a Borel subset
  $E_{x_0}'$ of $\R\times S^1$ such that $\P=\P_1\otimes \P_2\,$-a.s.,
\begin{equation}
\label{EF}
\{{\widetilde X}(x_0)\in B\}=\{({\widetilde Z}_\infty(x_0),\ {\widetilde A}_\infty(x_0))\in E_{x_0}'\}.
\end{equation}

For every $x\in M$, let $E_x$ be a Borel subset of $\R\times\R\times S^1$
such that
$$
{\{{\mathcal X}(x)\in B\}}\asequal\{({\mathcal Y}_\infty(x), {\mathcal
  Z}_\infty(x),{\mathcal A}_\infty(x))\in E_x\},
$$
where $\mathcal X(x)$ is a time-changed Brownian motion on $M$ started at $x$,
with coordinates $({\mathcal Y}_t(x)+t, {\mathcal Z}_t(x),{\mathcal
  A}_t(x)))$.  Since $B\in \SF_{\rm inv}$, we have
\begin{align*} 
  {\{{\widetilde X}(x_0)\in B\}}&={\{{\widetilde X}_{\tau+\newdot}(x_0)\in B\}},
\end{align*}
and hence 
$${\widetilde X}(x_0)\in B\quad\hbox{if and only if}\quad {\widetilde
  X}''({\widetilde X}'_\tau)\in B.$$
But by Lemma~\ref{P2}, $\P$-a.s.,
$$
\{{\widetilde X}(x_0)\in B\}= \{({\widetilde Y}_\infty(x_0), {\widetilde
  Z}_\infty(x_0),{\widetilde A}_\infty(x_0))\in E_{x_0}\},
$$
and for all $x\in S_{x_0}$, $\P_2\,$-a.s.,
$$
\{{\widetilde X}''(x)\in B\}= \{({\widetilde Y}''_\infty(x), {\widetilde
  Z}''_\infty(x),{\widetilde A}''_\infty(x))\in E_{x}\}.
$$
On the other hand, it is easy to see that
$$
{\widetilde Y}_\infty(x_0)(\om_1,\om_2)={\widetilde Y}''_\infty({\widetilde
  X}'_\tau(\om_1))(\om_2)-\tau(\om_1).
$$
Let $P_{{\widetilde X}'_\tau(x_0)}$ be the law of ${\widetilde
  X}'_{\tau}(x_0)$.  A consequence of the equalities above is that for
$P_{{\widetilde X}'_\tau(x_0)}$-almost all $x$, we have $\P(\,\newdot\,\vert {\widetilde
  X}_\tau(x_0)=x)$ a.s.,
\begin{align}
\label{EX}
\begin{split}
  &\{({\widetilde Y}''_\infty(x),
  {\widetilde Z}''_\infty(x),{\widetilde A}''_\infty(x))\in E_{x},\  {\widetilde X}_\tau(x_0)=x\}\\
  &\quad=\{{\widetilde X}''(x)(\om_2)\in B,\  {\widetilde X}_\tau(x_0)=x\}\\
  &\quad=\{{\widetilde X}(x_0)\in B,\  {\widetilde X}_\tau(x_0)=x\}\\
  &\quad=\{({\widetilde Y}_\infty(x_0), {\widetilde Z}_\infty(x_0),{\widetilde A}_\infty(x_0))\in E_{x_0},\  
   {\widetilde X}_\tau(x_0)=x\}\\
  &\quad = \{({\widetilde Y}''_\infty(x)(\om_2)-\tau(\om_1), {\widetilde
    Z}''_\infty(x) (\om_2),{\widetilde A}''_\infty(x)(\om_2))\in E_{x_0},\ 
  {\widetilde X}_\tau(x_0)=x\}.
\end{split}
\end{align}

On the other hand, we know that the law of $\tau$ given $ {\widetilde
  X}_\tau(x_0)$ and $\om_2$ has a positive density on $]0,\infty[$ and depends
only on $ {\widetilde X}_\tau(x_0)$. Hence equality of the second and
the last term in~\eqref{EX} implies that for 
$P_{{\widetilde X}'_\tau(x_0)}$-almost all $x$, $\P(\,\newdot\,\vert {\widetilde X}'_\tau(x_0)=x)$ a.s., if
$$
({\widetilde Y}''_\infty(x)(\om_2)-\tau(\om_1), {\widetilde
  Z}''_\infty(x)(\om_2),{\widetilde A}''_\infty(x)(\om_2))\in E_{x_0}
$$
then $E_{x_0}$ contains almost all $(y',{\widetilde
  Z}''_\infty(x)(\om_2),{\widetilde A}''_\infty(x)(\om_2))$ such that $y'\le
{\widetilde Y}''_\infty(x)(\om_2).$
Let
$$
E_{x_0}'=\{(z,a)\in \R\times S^1\mid \exists\, y\in\R \hbox{ such that }(y',z,a)\in E_{x_0} \text{for almost all
}\ y'\le y\}.
$$
We proved that for $P_{{\widetilde X}'_\tau(x_0)}$-almost all $x$,
$\P_1(\,\newdot\,\vert {\widetilde X}'_\tau(x_0)=x)\otimes \P_2$ a.s.,
\begin{align*}
  &\{{\widetilde X}''(x)(\om_2)\in B,\  {\widetilde X}_\tau(x_0)=x\}\\
  &\quad \quad\subset \{( {\widetilde Z}''_\infty(x)(\om_2),{\widetilde
    A}''_\infty(x)(\om_2))\in
  E'_{x_0},\  {\widetilde X}_\tau(x_0)=x\}\\
  &\quad \quad= \{( {\widetilde Z}_\infty(x_0)(\om),{\widetilde
    A}_\infty(x_0)(\om))\in E'_{x_0},\ {\widetilde X}_\tau(x_0)=x\}.
\end{align*}
Since
$$
\{{\widetilde X}(x_0)(\om_1,\om_2)\in B,\ {\widetilde
  X}_\tau(x_0)(\om_1)=x\}=\{{\widetilde X}''(x)(\om_2)\in B,\ {\widetilde
  X}_\tau(x_0)(\om_1)=x\},
$$
we get almost surely,
\begin{equation}
\label{EFI}
\{{\widetilde X}(x_0)\in B\}\subset \{( {\widetilde Z}_\infty(x_0),{\widetilde A}_\infty(x_0))\in E'_{x_0}\}.
\end{equation}
To establish the other inclusion, let us define, for $(z,a)\in E'_{x_0}$, 
\begin{align*}
  f(z,a)=
\begin{cases}
\sup\{y\in \R, \ \hbox{ for almost all }\ y'\le y,\ (y',z,a)\in E_{x_0}\},\\
  -\infty \ \hbox{ if the set is empty.}
\end{cases}
\end{align*}
We know that on $\{{\widetilde X}(x_0)\in B\}$, for $P_{{\widetilde
    X}'_\tau(x_0)}$-almost all $x$, $\P(\,\newdot\,\vert {\widetilde
  X}'_\tau(x_0)=x)$ a.s., $$f({\widetilde Z}_\infty(x_0),{\widetilde
  A}_\infty(x_0))\in\R\cup\{\infty\}.$$
Recall that ${\widetilde
  Y}_\infty(x_0)(\om_1,\om_2)={\widetilde Y}''_\infty({\widetilde
  X}'_\tau(x_0)(\om_1))(\om_2)-\tau(\om_1)$, and that
\begin{align*}
  &\{({\widetilde Z}_\infty(x_0)(\om_1,\om_2),{\widetilde A}_\infty(x_0)(\om_1,\om_2))\in E'_{x_0},\ {\widetilde X}_\tau(x_0)(\om_1,\om_2)=x\}\\
  &\quad =\{({\widetilde Z}''_\infty(x)(\om_2),{\widetilde
    A}''_\infty(x)(\om_2))\in E'_{x_0},\ {\widetilde X}'_\tau(x_0)(\om_1)=x\}.
\end{align*}
By the positiveness of the density of the conditional law of $\tau$, for
$P_{{\widetilde X}'_\tau(x_0)}$-almost all $x$ and $\om_2\in \Om_2$ such that
$$({\widetilde Z}''_\infty(x)(\om_2),{\widetilde A}''_\infty(x)(\om_2))\in
E'_{x_0},$$
we find
$$
\P_1\left\{\tau\ge {\widetilde Y}''_\infty(x)(\om_2)-f( {\widetilde
    Z}''_\infty(x)(\om_2),{\widetilde A}''_\infty(x)(\om_2))\ \vert\ 
  {\widetilde X}_\tau(x_0)=x\right\}>0.
$$
This yields
$$
\P_1\left\{({\widetilde Y}''_\infty(x)(\om_2)-\tau, {\widetilde
    Z}''_\infty(x)(\om_2),{\widetilde A}''_\infty(\om_2))\in E_{x_0} \ \vert\ 
  {\widetilde X}_\tau(x_0)=x\right\}>0
$$
which implies by~\eqref{EX} that
$$
\P_1\left\{{\widetilde X}''(x)(\om_2)\in B\vert\ {\widetilde
    X}_\tau(x_0)=x\right\}>0.
$$
Taking into account that above probability takes its values in $\{0,1\}$, it must be equal to
$1$. Consequently for $P_{{\widetilde X}'_\tau(x_0)}$-almost all $x$,
$\P(\,\newdot\,\vert {\widetilde X}_\tau(x_0)=x)$-almost surely,
\begin{align*}
  \{{\widetilde Z}''_\infty(x)(\om_2),{\widetilde A}''_\infty(x)(\om_2))\in
  E'_{x_0},\ {\widetilde X}'_\tau(x_0)=x\}\subset \{{\widetilde X}''(x)\in B,\ 
  {\widetilde X}'_\tau(x_0)=x\}
\end{align*}
which rewrites as
\begin{align*}
  \{{\widetilde Z}_\infty(x_0),{\widetilde A}_\infty(x_0))\in E'_{x_0},\ 
  {\widetilde X}_\tau(x_0)=x\}\subset \{{\widetilde X}(x_0)\in B,\ {\widetilde
    X}_\tau(x_0)=x\}.
\end{align*}
We get, $\P$-almost surely,
\begin{equation}
\label{ESI}
\{{\widetilde Z}_\infty(x_0),{\widetilde A}_\infty(x_0))\in E'_{x_0}\}\subset \{{\widetilde X}(x_0)\in B\}.
\end{equation}
With \eqref{EFI} and \eqref{ESI} we finally obtain~\eqref{EF}: 
$$
\{{\widetilde Z}_\infty(x_0),{\widetilde A}_\infty(x_0))\in
E'_{x_0}\}=\{{\widetilde X}(x_0)\in B\},\quad \text{$\P$-a.s.}
$$
This achieves the proof.
\end{proof}
Putting together Theorems~\ref{zisnotconstant} and~\ref{T1} we are now able to
state our main result, giving a complete characterization of the Poisson
boundary of $M$.
\begin{thm}
\label{PBM}
Let $\SB(\R\times S^1)$ be the set of bounded measurable functions on
$\R\times S^1$, endowed with the equivalence relation
$f_1\simeq f_2$ if $f_1=f_2$ almost everywhere with respect to the Lebesgue measure.

The map
\begin{align*}
  \left(\SB(\R\times S^1)/\simeq\right)&\longrightarrow \SH_b(M)\\
  f&\longmapsto \big(x\mapsto
    \E\left[f\left(Z_\zeta(x),A_{\zeta}(x)\right)\right]\big)
\end{align*}
is one to one. More precisely, the inverse map is given as follows.  For $x\in
M$, letting $K(x,\newdot,\newdot)$ be the density of $(Z_\zeta(x),
A_\zeta(x))$ with respect to the Lebesgue measure on $\R\times S^1$, for all
$h\in \SH_b(M)$ there exists a unique $f\in\SB(\R\times S^1)/\simeq$ such that
$$
\forall x\in M,\quad h(x)=\int_{\R\times S_1}K(x,z,a)f(z,a)\,dzda.
$$
Moreover, for all $x\in M$, the kernel $K(x,\newdot,\newdot)$ is almost
everywhere strictly positive.
\end{thm}


\providecommand{\bysame}{\leavevmode\hbox to3em{\hrulefill}\thinspace}

\end{document}